\documentclass[11pt]{amsart}
\usepackage{graphicx}
\usepackage{amsmath}
\usepackage{amssymb}
\usepackage{amsthm}
\usepackage{mathtools}
\usepackage{mathscinet}
\usepackage{tikz-cd}
\usepackage{color}
\usepackage{enumitem}
\usepackage{tensor}
\usepackage{faktor}
\usepackage{mathrsfs}
\usepackage{float}
\usepackage[]{hyperref}
\hypersetup{
	pdftitle={Large Deviations and Effective Equidistribution},
	pdfauthor={Ilya Khayutin},
	bookmarksnumbered=true,     
	bookmarksopen=true,         
	bookmarksopenlevel=1,       
	colorlinks=true,            
	pdfstartview=Fit,           
	pdfpagelayout=TwoPageRight
}

\makeatletter
\def\paragraph{
	\@startsection{paragraph}{4}
	\z@{.5\linespacing\@plus.7\linespacing}{-.5em}%
	{\normalfont\itshape}}
\makeatother

\newcommand*{\lfaktor}[2]{
  \raisebox{-0.5\height}{\ensuremath{#1}}
  \mkern-2mu\backslash\mkern-3mu
  \raisebox{0.5\height}{\ensuremath{#2}}
}

\newcommand*{\sdfaktor}[3]{
  \raisebox{-0.5\height}{\ensuremath{#1}}
  \mkern-2mu\backslash\mkern-3mu
  \raisebox{0.5\height}{\ensuremath{#2}}
  \mkern-3mu\slash\mkern-2mu
  \raisebox{-0.5\height}{\ensuremath{#3}}
}

\renewcommand{\restriction}{\mathord{\upharpoonright}}
\newcommand{\dif}{\,\mathrm{d}}

\DeclareMathOperator*{\Id}{Id}

\DeclareMathOperator*{\Ad}{Ad}

\DeclareMathOperator*{\vol}{vol}

\DeclareMathOperator*{\Gal}{Gal}

\DeclareMathOperator*{\diag}{diag}

\DeclareMathOperator*{\rank}{rank}
\DeclareMathOperator*{\esssup}{ess\,sup}
\DeclareMathOperator*{\essinf}{ess\,inf}

\newtheorem{thm}{Theorem}
\numberwithin{thm}{section} 
\newtheorem*{thm*}{Theorem}
\newtheorem{lem}[thm]{Lemma}
\newtheorem{prop}[thm]{Proposition}
\newtheorem{cor}[thm]{Corollary}

\theoremstyle{definition}
\newtheorem{defi}[thm]{Definition}
\newtheorem{exmpl}[thm]{Example}
\theoremstyle{remark}
\newtheorem{remark}[thm]{Remark}

\title{Large Deviations and Effective Equidistribution}
\author[I. Khayutin]{Ilya Khayutin}
\address{The Einstein Institute of Mathematics\\
	Edmond J. Safra Campus, The Hebrew University of Jerusalem, Givat Ram,
	Jerusalem, 9190401, Israel}

\begin{document}

\setcounter{tocdepth}{1}

\begin{abstract}
We study large deviations for measurable averaging operators on state spaces of dynamical systems.
Our main motivation is the Hecke operators on the modular curve $Y_0(p^n)$ and their generalization to higher rank $S$-arithmetic quotients. 

We prove a relatively sharp large deviations result in terms of the norm of the averaging operator restricted to the orthogonal complement of the constant functions in $L^2$. In the self-adjoint case this norm is expressible by the spectral gap.

Developing ideas of Linnik and Ellenberg, Michel and Venkatesh, we use this large deviation result to prove an effective equidistribution theorem on a state space.
The novelty of our results is that they apply to measures with sub-optimal bounds on the mass of Bowen balls.

We present two new applications to our effective equidistribution result. The first one is effective rigidity for the measure of maximal entropy on $S$-arithmetic quotients with respect to a semisimple action in a non-archimedean place. Measures having large enough metric entropy must also be close on the state space to the Haar measure.
This is a partial extension of a recent result of R\"{u}hr to a significantly more general setting. 

The second one is non-escape of mass for sequences of measures having large entropy with respect to a semisimple element in a non-archimedean place. This generalizes similar known results for real flows. Our methods differ from the methods used by R\"{u}hr and in the previously known non-escape of mass results.
\end{abstract}
\maketitle

\bgroup
\hypersetup{linkcolor = blue}
\tableofcontents
\egroup

\section{Introduction}
\subsection{State Spaces and Averaging Operators}
In this paper we deal with state spaces for measurable dynamical system. A dynamical systems for us is  a measure space $(Y,\mathcal{Y},m)$ equipped with a not necessarily invertible measure-preserving transformation $S\colon Y\to Y$. A state space for the system $Y$ is a measurable space $(X,\mathcal{X})$ with a measurable map $\pi\colon (Y,\mathcal{Y})\to(X,\mathcal{X})$. The space $X$ has a natural associated measure $\pi_*m$ -- the pushforward of $m$; yet the action of $S$ does not necessarily descend to $X$.

One example to keep in mind is Markov shifts on  finite alphabets and the associated directed graphs.

We are mainly interested  in a family of dynamical systems arising in the $S$-arithmetic setting. Let $\mathbf{G}$ be a reductive linear algebraic group defined over $\mathbb{Q}$. Let $S$ be a finite set of places of $\mathbb{Q}$ including $\infty$ such that $G_S\coloneqq \prod_{v\in S} \mathbf{G}(\mathbb{Q}_v)$ is non-compact. Fix $\Gamma<G_S$ a congruence lattice and set $Y\coloneqq\lfaktor{\Gamma}{G_S}$. 

Let $p\in S$ be a \emph{finite} place and $a\in G_p\coloneqq\mathbf{G}(\mathbb{Q}_p)$ a semisimple element not belonging to a compact group. 
The space $Y$ is equipped with a unique Borel probability measure invariant under the right action of $G_S$ -- this is the Haar measure $m$. The dynamical system we study is $Y$ equipped with the right action by $a$. A compact subgroup $K<G_S$ defines a state space $X\coloneqq\faktor{Y}{K}$ for this dynamical system.

Lets present an important example which we address in this paper.
The space $X\coloneqq\lfaktor{\mathbf{PGL}_2(\mathbb{Z})}{\mathbf{PGL}_2(\mathbb{R})}$ can be seen as a quotient by a compact group of a larger $S$-arithmetic locally homogeneous space, specifically
\begin{align*}
Y &\coloneqq \lfaktor{\mathbf{PGL}_2(\mathbb{Z})}{\mathbf{PGL}_2(\mathbb{R})\times \mathbf{PGL}_2(\mathbb{Q}_p)}\\
X &\cong \faktor{Y}{\mathbf{PGL}_2(\mathbb{Z}_p)} 
\end{align*}
If we consider $Y$ as a dynamical system with the right action by $a=\begin{pmatrix} p&0\\0&1\end{pmatrix}$, then the action does not descent to the space $X$. Nevertheless, a ghost of this action is still visible in $X$ --  the $p$-Hecke operator. Our main focus is higher rank generalization of this example. 

For a general dynamical system and a state space we show how the Koopman operator $U_S$ on $L^2(Y,m)$ induces an averaging operator $T$ on $L^2(Y,m)$ defined by
\begin{equation*}
T\,f\coloneqq\mathbb{E}\left(U_S\,\mathbb{E}\left(f \mid \mathcal{X}\right) \mid \mathcal{X}\right)
\end{equation*}
A main feature of $T$ is that it can be restricted to an operator from $L^2(X,m)$ to itself. 

\subsection{Large Deviations}
We study state spaces for which the averaging operator $T$ behaves in a suitable sense like the transition operator of a Markov chain. We provide the formal definition of the required property, which we call property (M), in Definition \ref{def:M}. Notice that the example of the modular surface  introduced above does not possess property (M), yet this is amended by dividing by an Iwahori subgroup of $\mathbf{PGL}_2(\mathbb{Q}_p)$ instead of a maximal compact one. One can then deduce all the result we prove on the basic modular surface from the analogues results on this finer state space.   

We develop a large deviation bound for empirical averages of functions on the state space in terms of $\|T\|_{L^2_0(X,m)}$, where $L^2_0(X,m)$ is the orthogonal complement to the constant functions in $L^2(X)$. If the averaging operator $T$ is self-adjoint, which is not always the case, then $1-\|T\|_{L^2_0(X,m)}$ is the spectral gap of $T$. Here is a slightly simplified version of our large deviation result.
\begin{thm*}[Theorem \ref{thm:LD-specgap}]
Let $Y$ be a dynamical system with a state space $X$ having property (M). Denote $\lambda\coloneqq\|T\|_{L^2_0(X,m)}\leq1$.

Let $\varphi\colon Y\to[0,1]$ be a $\mathcal{X}$-measurable function and $\eta\geq m(\varphi)$. For any $\theta\in\left\{\lambda^k\mid k\in\mathbb{N}\right\}$ such that  $\theta\leq\frac{\eta-m(\varphi)}{1-m(\varphi)}$ 
\begin{align*}
m&\left(y\mid \frac{1}{n}\sum_{i=0}^{n-1} \varphi(S^n y)\geq \eta\right)\\
&\ll_{\theta,\lambda,\eta,m(\varphi)}
\exp\left[-n(-\log\lambda)\frac{D(\eta \,\|\, \theta+(1-\theta)m(\varphi))}{-\log\theta}\right]
\end{align*}
Where for any $0\leq a,b \leq 1$ the function $D(a\,\|\,b)$ is the binary Kullback-Leibler divergence in natural base defined by
\begin{equation*}
D(a\,\|\,b)\coloneqq -a\log\frac{b}{a}-(1-a)\log\frac{1-b}{1-a}
\end{equation*}
\end{thm*}
This bound applies in particular to non-reversible Markov chains in discrete time. Large deviation bounds for Markov chains with decay rate depending on spectral data of the transition operator is a well-studied subject. It has been initiated by Gillman \cite{Gillman} and developed, inter alios, by \cite{AFWZ, KahaleLD, Lezaud, LeonPerron, Wagner, CLLM, PointsOnSphere}. 

This form for the rate function of the large deviation estimate is new already in the setting of Markov chains. In particular, our large deviation theorem has the pleasant feature that for $\lambda=0$, which holds for i.i.d.\ variables taking values in a finite alphabet, we recover the Chernoff-Hoeffding bound \cite{Hoeffding} in full strength.

\subsection{Effective Rigidity of The Measure of Maximal Entropy}\label{sec:intro-max}
Suppose we have an $S$-arithmetic quotient $\lfaktor{\Gamma}{G_S}$ equipped with the right action by a semisimple element in a non-archimedean place $a\in G_p$. As long as the group is generated by the horospherical subgroups corresponding to $a$ -- the Haar measure is known to be the unique measure of maximal entropy \cite[\S 9]{MargulisTomanov}, \cite[\S 7]{Pisa}.

We have the following form effective rigidity for the measure of maximal entropy.
\begin{thm*}[Theorem \ref{thm:rigidity-max-alg}]
Let $\mathbf{G}$ be a connected simply connected absolutely almost simple linear algebraic group defined over $\mathbb{Q}$. Let $S$ be a finite set of places for $\mathbb{Q}$ including $\infty$, such that $G_S\coloneqq \prod_{v\in S} \mathbf{G}(\mathbb{Q}_v)$ is not compact. 

Let $p\in S$ be a finite place such that $G_p\coloneqq\mathbf{G}(\mathbb{Q}_p)$ is not compact. Fix $a\in G_p\cap\mathbf{G}(\mathbb{Q})$ $\mathbb{Q}_p$-regular and semisimple and not contained in a compact subgroup. 

Denote by $A<G_p$ the maximal split subtorus of a maximal torus including the element $a$. Let $\Delta$ be the affine Bruhat-Tits building associated to $G_p$.
Denote by $H<G_p$ the pointwise stabilizer of the apartment corresponding to $A$ in $\Delta$. 

Let $\Gamma<G_S$ be a congruence lattice and denote by $m$ the Haar probability measure on $\lfaktor{\Gamma}{G_S}$.

Let $\varphi\colon\sdfaktor{\Gamma}{G_S}{H}\to[0,1]$ be a continuous function that is smooth in the $p$-adic coordinate, i.e.\ when considered as a function on $\lfaktor{\Gamma}{G_S}$ it is invariant under some compact open subgroup $K_0<G_p$ containing $H$. 

Then for any $\varepsilon>0$ there is an explicitly computable constant $C>0$ depending only on $\varepsilon$, $G_p$, $K_0$ and $a\in G_p$ such that for any $a$-invariant measure $\mu$ we have
\begin{equation*}
|\mu(\varphi)-m(\varphi)|\leq C \left(h_m(a)-h_\mu(a)\right)^{1/2-\varepsilon}
\end{equation*}
\end{thm*}

This should be compared with the recent theorem of R\"{u}hr \cite[Theorem 1.1]{Ruhr}. The result of R\"{u}hr applies to the special case when\footnote{Necessarily $\lfaktor{\Gamma}{G_S}$ is a compact quotient.} $G_S=\mathbf{G}(\mathbb{Q}_p)$, $\mathbf{G}$ is $\mathbb{Q}_p$-split and $a$ is a semisimple element that does not belong to a compact subgroup. Wherever it applies \cite[Theorem 1.1]{Ruhr} is somewhat stronger then Theorem \ref{thm:rigidity-max-alg} and we discuss the major differences in the following.

Analogues results to R\"{u}hr's and ours are known also in other settings \cite{Polo, KadyrovP}.

We note that Theorem \ref{thm:rigidity-max-alg} is derived from stronger results of more technical nature\footnote{If the measure $\mu$ is ergodic we have a relatively concise form of the strong version of our results in Corollary \ref{cor:entropy-ergodic}.}. In particular, the methods which lead to the proof of Theorem \ref{thm:rigidity-max-alg} allow us to prove the non-escape of mass results described in the next section. Both Theorem \ref{thm:rigidity-max-alg} and a possible extension of \cite[Theorem 1.1]{Ruhr} to general $S$-arithmetic quotients would imply qualitatively weaker results for non-escape of mass.

Theorem \ref{thm:rigidity-max-alg} and \cite[Theorem 1.1]{Ruhr} both rely eventually on exponential decay of matrix coefficients.  Bounds on matrix coefficients enter our argument because it implies a bound on the $L^2_0$-norm of the averaging operator \cite{SarnakKyoto, Chiu, BurgerSarnak, CU, COU, GMO}.
R\"{u}hr uses decay of matrix coefficients to show a uniform effective equidistribution under the action of the stable horospherical subgroup corresponding to $a$; equidistribution under the horospherical action is related to  mixing of the $a$-action as has been already observed by Margulis \cite{Margulis}. The pertinent bounds on matrix coefficients are the result of the combined work of \cite{BurgerSarnak,CU,ClozelTau,GelbartJacquet,JacquetLanglands,Rogawski, OhBounds, Oh}.

The primary strength of our result is that it covers significantly more cases. Most importantly, we are able to treat $S$-arithmetic quotients with an archimedean part.

The central shortcoming of our result is that it applies only to functions invariant under $H$. The theorem of R\"{u}hr applies to all smooth functions. The reason we are restricted to $H$ invariant functions is quite natural. To deduce Theorem \ref{thm:rigidity-max-alg} we need a state space $X=\faktor{Y}{K}$ such that $\varphi$ factors through $X$. The  state spaces come from subgroups $K<G_p$ that stabilize two chambers\footnote{Not necessarily distinct.} in the apartment corresponding to $A$ in $\Delta $, see \S\ref{sec:walks-buildings}. In particular, any state space that we analyze must be coarser then the quotient by the pointwise stabilizer of this apartment -- $H$.

Another difference is that our result applies to functions in $L^\infty$. If we take $\varphi$ bounded but such that we do not have necessarily $0\leq\varphi\leq 1$ then  $|\mu(\varphi)-m(\varphi)|$ will be proportionate to $\|\varphi\|_\infty$. The result of \cite{Ruhr} is for any smooth function in $L^2$ and more importantly the bound is proportionate to  $\|\varphi\|_2$. The reason our result depends on the $L^\infty$ norm lies in the basis of the development of the large deviations bound in \S\ref{sec:Kahale}. 

Last, we lose a factor of $\varepsilon$ in the exponent compared to \cite[Theorem 1.1]{Ruhr}. We note that we derive Theorem \ref{thm:rigidity-max-alg} from stronger results, specifically Proposition \ref{prop:limit-tight}. 

\subsection{Non-Escape of Mass}\label{sec:intro-non-escape}
The relation between measures of high entropy and non-escape of mass has been introduced by Einsiedler, Lindenstrauss, Michel and Venkatesh \cite{ELMVPGL2} and extended in \cite{EK}, \cite{EKP}, \cite{KadyrovH} and \cite{KKLM}. Closely related is the work of Cheung \cite{Cheung}.
These results described non-escape of mass for flows under a one-parameter subgroup in the real place\footnote{Actions in rank $1$ in a $p$-adic place have been briefly discussed in \cite[Remark 5.2]{ELMVPGL2}.}.

We present a theorem that is closely related to the previously studied cases. Our result establishes non-escape of mass for measures of high entropy invariant under the action of a regular semisimple element in the $p$-adic place. The tools we use in proving this theorem are considerably different from the methods applied in the proofs of the previously established results. Specifically, none of the above relied on decay of matrix coefficients.

Assume that $a\in G_p$ is $\mathbb{Q}_p$-regular and semisimple. Let $A<G_p$ be the maximal split subtorus of a maximal torus containing $a$ and denote by $A^+$ the closed Weyl chamber corresponding to $a$. The function $\eta\colon A\to\mathbb{R}$ has been defined by Oh \cite{Oh} and extended in \cite{COU} and \cite{GMO}. It is equal to the product of all the roots in a maximal strongly orthogonal system of positive roots, see Definition \ref{def:eta-function}. This function is closely related to the decay of matrix coefficients and the norm gap of Hecke operators. 

For the diagonal torus in $\mathbf{SL}_n$ and
\begin{equation*}
A^+=\left\{\diag(a_1,\ldots,a_n) \mid a_i\in\mathbb{N}, a_i>0 \textrm{ and }\forall 1\leq i\leq n-1\colon a_{i+1}|a_i  \right\}
\end{equation*}
the function $\eta$ for $a\in A^+$ is equal to
\begin{equation*}
\eta(a)=\prod_{i=1}^{\lfloor n/2 \rfloor } \frac{|a_i|_p}{|a_{n+1-i}|_p}
\end{equation*}
 
\begin{thm*}[Theorem \ref{thm:non-escape}]
Let $\mathbf{G}$ be a connected simply connected absolutely almost simple linear algebraic group defined over $\mathbb{Q}$. Let $S$ be a finite set of places for $\mathbb{Q}$ including $\infty$, such that $G_S\coloneqq \prod_{v\in S} \mathbf{G}(\mathbb{Q}_v)$ is not compact. 

Let $p\in S$ be a finite place such that $G_p\coloneqq\mathbf{G}(\mathbb{Q}_p)$ is not compact, set $r=2$ if $\rank_{\mathbb{Q}_p} \mathbf{G}=1$ and $r=1$ if the rank is higher. 

Fix $a\in G_p\cap\mathbf{G}(\mathbb{Q})$ a $\mathbb{Q}_p$-regular and semisimple element not belonging to a compact subgroup. Let $A<G_p$ be the maximal split subtorus of a maximal torus containing $a$, set $A^+$ to be the closed Weyl chamber corresponding to $a$. Define $\eta$ with respect to these choices.  

Let $\Gamma<G_S$ be a congruence lattice and denote by $m$ the Haar probability measure on $\lfaktor{\Gamma}{G_S}$.

Suppose we are given a sequence of $a$-invariant probability measures $\mu_i$ on $Y\coloneqq\lfaktor{\Gamma}{G}$. Set $h=\liminf_{i\to\infty} h_{\mu_i}(a)$.
If $\mu_i$ converges to a measure $\mu$ in the weak-$*$ topology then
\begin{equation*}
\mu(Y)\geq 1-2r\frac{h_m(a)-h}{-\log\eta(a)}
\end{equation*}
\end{thm*}
The factor $r=2$ in rank $1$ is related to the Ramanujan conjecture \cite{CU} and is conjectured to be redundant \cite[Conjecture 2.15]{GMO}. 

Although our results are for the $p$-adic place our expression for the bound on the mass is well defined in the real place as well. It is interesting to compare this bound with the established bounds in the real place and especially with \cite{KadyrovEscape}, \cite{KadyrovPohl} which establish that some of the known bounds for real flows are sharp. 

In the known archimedean cases with $\rank>1$ the expression $1-2\frac{h_m(a)-h}{-\log\eta(a)}$ is of the right order of magnitude but weaker then the known results.

Lets consider the the diagonal flow 
\begin{equation*}
a(t)=\diag(e^t,e^t,\ldots,e^{-nt})
\end{equation*}
on $\lfaktor{\mathbf{SL}_{n+1}(\mathbb{Z})}{\mathbf{SL}_{n+1}(\mathbb{R})}$. It is established in \cite{KKLM} that if sequence of $a(t)$ invariant measures $\mu_i$ converges to $\mu$ then $\mu(Y) \geq h/n-n$ where $h=\liminf_{i\to\infty} h_{\mu_i}(a(1))$. Moreover, Kadyrov \cite{KadyrovEscape} shows that this bound is actually sharp. The value of $1-2r\frac{h_m(a)-h}{-\log\eta(a)}$ in this case is $2\left(\frac{h}{n+1}-n\right)+1$ which is strictly smaller then the sharp bound of \cite{KKLM}.

\subsection{Effective Equidistribution}
The bridge between large deviation estimates and our results 
for $S$-arithmetic quotients is provided by an effective equidistribution result, Theorem \ref{thm:effective}. This theorem is closely related to the work of Linnik \cite{LinnikBook} and Ellenberg, Michel and Venkatesh \cite{PointsOnSphere}. 

The method of proof for Theorem \ref{thm:effective} uses the large deviation bound to show that if a measure $\mu$ on an $S$-arithmetic quotient $\lfaktor{\Gamma}{G_S}$ has on average small Bowen balls for the right action by $a$, then on the state space $\sdfaktor{\Gamma}{G_S}{K}$ the measure $\mu$ must be close in a suitable sense to the Haar measure. The novelty of our result compared to \cite{LinnikBook} and \cite{PointsOnSphere} is that we are able to provide non-trivial information even for a sub-optimal bound on the average measure of Bowen balls. 

In the works of Linnik and Ellenberg, Michel, and Venkatesh the Haar measure of a set of atypical points is bounded above by an exponentially decaying bound coming from a large deviation estimate. The same measure is bounded below by a sub-exponential bound derived from an asymptotically optimal estimate on the average size of Bowen balls.
The inequality between an exponential and a sub-exponential expression provides a contradiction no matter how good the rate function is in the large deviations estimate, as long as we have exponential decay.

When the available estimate on the average measure of Bowen balls is sub-optimal the lower bound is no longer sub-exponential. In this case we have lower an upper bounds both of exponential type. In this case, to derive a contradiction we need to compare the rate of decay of the two bounds. For this we require that the rate of the large deviation estimate be as good as possible.

The better exponential rate we have in the pertinent large deviations bound the better equidistribution result we have using the methods of Theorem \ref{thm:effective}.

\subsection{Acknowledgments}
This paper is part of the author's PhD thesis conducted at the Hebrew University of Jerusalem under the guidance of Prof. E. Lindenstrauss, to whom I am grateful for introducing me to homogeneous dynamics and number theory, and for many helpful discussions and insights. I would like to express my gratitudes to Akshay Venkatesh for encouraging and valuable conversations. It is my pleasure to thank Or Landesberg and Shai Evra for helpful discussions and much needed coffee breaks.

Last but not least, I would like to thank my wife, Olga Kalantarov, for her enduring support during the preparation of this paper.

The author has been supported by the European Research Council [AdG Grant 267259] throughout his PhD studies.

\section{Averaging Operators}
In this section we present a general framework for studying contraction properties for averaging operators coming from dynamical systems.

First we describe how a measurable dynamical system and a state space gives rise to an averaging operator.
\begin{defi}
Let $(Y,\mathcal{Y},m,S)$ be a dynamical system, viz.\ $(Y,\mathcal{Y},m)$ is a measure space and $S\colon Y\to Y$ is a measure preserving transformation. We define a \emph{state space} for the system $Y$ to be a measurable space $(X,\mathcal{X})$ with a measurable map $\pi:Y\to X$. The state space has a natural measure on it $\pi_*m$. 

We identify $\mathcal{X}$ with the $\sigma$-algebra $\pi^{-1}\mathcal{X}\subseteq\mathcal{Y}$. State spaces for $Y$ are in bijection with $m$-equivalence classes of sub-$\sigma$-algebras of $\mathcal{Y}$. Notice that $S$ doesn't necessarily act on $X$, equivalently $\mathcal{X}$ is not necessarily $S$-stable.
\end{defi}

\begin{defi}
Let $U_S$ be the Koopman operator associated with $S$ acting either on $L^1(Y)$ or on $L^2(Y)$. The averaging operator $T_S$ is defined on $L^1(Y)$ and on $L^2(Y)$ by $T_S\,f=\mathbb{E}\left(U_S\,\mathbb{E}\left(f\mid \mathcal{X}\right) \mid \mathcal{X}\right)$. If $P_X$ is the orthogonal projection onto $L^2(X)\leq L^2(Y)$ then $T_S$ on $L^2(Y)$ can be written as $T_S=P_X U_S P_X$. In functional analysis literature the operator $T_S$ on $L^2(Y)$ is called the compression of $U_S$ onto $L^2(X)$. In particular, $T_S$ is a well defined operator on $L^1(X)$ and $L^2(X)$.
\end{defi}

\begin{exmpl}
A finite state Markov chain with the shift map is a dynamical system. The associated directed graph is a state space for this dynamical system. The averaging operator in this case is the weighted normalised adjacency matrix. 
\end{exmpl}
\begin{exmpl}\label{exmpl2}
The space
\begin{equation*}
Y=\lfaktor{\mathbf{PGL}_2\left(\mathbb{Z}\left[\frac{1}{p}\right]\right)}  {\mathbf{PGL}_2\left(\mathbb{R}\right) \times \mathbf{PGL}_2\left(\mathbb{Q}_p\right)} 
\end{equation*}
is a dynamical system when equipped with the $p$-adic diagonal right action by $\begin{pmatrix} p&0 \\ 0&1 \end{pmatrix}$.
The space
\begin{equation*}
X=\lfaktor{\mathbf{PGL}_2(\mathbb{Z})} {\mathbf{PGL}_2(\mathbb{R})}
\end{equation*}
can be identified with $\faktor{Y}{K}$ where $K=\mathbf{PGL}_2(\mathbb{Z}_p)$. In particular, there is a continuous projection $Y\to X$ and we consider $X$ as a state space for the dynamical system $Y$.

The associated averaging operator is the well-known $p$-Hecke operator. 
\end{exmpl}
\begin{exmpl}
Lets consider a compact Riemannian manifold $M$ and let $g_t$ be the geodesic flow on $T_1\,M$. Fix $t\in\mathbb{R}$, the dynamical system we are interested in is the $t$-time map of the geodesic flow on $T_1\,M$ with the Liouville measure. The state space is $M$ and $T_{g_t}$ is closely related to the Laplacian which is up to a constant equal to $\lim_{t\to 0} \frac{1}{t^2}(T_{g_t}-Id)$.
\end{exmpl}

Those examples demonstrate that although the Koopman operator being a unitary operator is as far as an operator can be from having a spectral gap, its compression to an averaging operator on a state space can have a spectral gap. A necessary but not a sufficient condition for this is that the state space's $\sigma$-algebra has no $S$-stable sub-$\sigma$-algebra.

More specifically, we are going to discuss not the spectral gap of the operator $T_S$ but rather its $L^2_0$-norm. The space $L^2_0(Y,m)$ is the orthogonal complement to the constant function in $L^2(Y,m)$. If $T_S$ is self-adjoint then
the quantity 1-$\|T_S\|_{L^2_0(Y,m)}$ is the spectral gap of $T_S$.

\subsection{Property (M)}

When the dynamical system $Y$ is a 1-step Markov process with the associated directed graph as a state space, it holds for all $n\in\mathbb{N}$ that $T_{S^n}={(T_S)}^n$, i.e.\ the operators $\left\{T_{S^n}\right\}_{n\in\mathbb{N}}$ form a semigroup.

It is important to note that it is not always true that the averaging operators form a semigroup.
They do not form a semigroup\footnote{It is well known for $p$-Hecke operators on the modular surface that $T_{p^2}\neq (T_p)^2$} in Example \ref{exmpl2}, but they almost do so. We need to divide on the right not by a maximal compact subgroup of $\mathbf{PGL}_2(\mathbb{Q}_p)$ but by an Iwahori subgroup. The associated state space is $\lfaktor{\Gamma}{\mathbf{PGL}_2(\mathbb{R})}$ where $\Gamma$ is a congruence lattice defined by $\Gamma=\left\{\begin{pmatrix}
a & b\\ c & d
\end{pmatrix} \in \mathbf{PGL}_2(\mathbb{Z}) \mid c\equiv 0 \mod p \right\}$.

The primary property we use for the derivation of a large deviations bound is described in the following definition. This property is possessed by 1-step Markov processes.

\begin{defi}\label{def:M}
Let $(Y,\mathcal{Y},m,S)$ be a dynamical system. Let $(X,\mathcal{X})$ be a state space for $Y$.
We say that the system $Y$ has \emph{property (M)} with respect to the state space $X$ if there exists a closed $U_S$-stable $L^\infty(X)$-module\footnote{A closed $L^\infty(X)$-module in $L^2(Y)$ is a closed subspace of $L^2(Y)$ stable under multiplication by any member of $L^\infty(X)$.} $\mathcal{M}\subseteq L^2(Y)$ such that $L^2(X)\subseteq \mathcal{M}$ and
\begin{equation*}
(P_X U_S  P_X)\restriction_{\mathcal{M}}= (P_X U_S)\restriction_{\mathcal{M}}
\end{equation*}
Equivalently, the orthogonal complement of $L^2(X)$ in $\mathcal{M}$ is $U_S$-stable.
\end{defi}

\begin{remark}
Property (M) for $S$ implies the same for $S^k$ for any $k\in\mathbb{N}$.
\end{remark}
\begin{remark}
Assume the transformation $S$ is invertible, i.e.\ we have a dynamical system with a $\mathbb{Z}$-action.

In the interesting case that the state space is not invariant under the action of $S$, if property (M) holds  then on $\mathcal{M}$ the Koopman operator $U_S$ would be non-invertible, while on $L^2(Y)$ as a whole it is invertible.
If $U_S$ were invertible over $\mathcal{M}$, then the $U_S$-stability of the orthogonal complement of $L^2(X)$ in $\mathcal{M}$ would imply the
$U_S$-stability of $L^2(X)$ which is equivalent to $X$ being $S$-invariant.
\end{remark}

Dynamical systems with property (M) over a state space have the following important properties.
\begin{prop}\label{prop:M-basic-properties}
Let $(Y,\mathcal{Y},m,S)$ be a dynamical system and $(X,\mathcal{X})$ a state space for $Y$ with property (M).
\begin{enumerate}[leftmargin=0pt, itemindent=*]
\item
$T_{S^n}=(T_S)^n$ for all $n\geq1$
\item For $h\in L^\infty(X)$ define $M_h$ to be the operator over $L^2(Y)$ of multiplying by h. Then
$P_X (M_h U_S M_h)^n P_X= (M_h T_S M_h)^n$ for all $n\geq1$
\end{enumerate}
\end{prop}
\begin{proof}
The second claim is a generalization of the first so it is enough to prove it. 

Let $\mathcal{M}\leq L^2(Y)$ be as in the definition of property (M). For each $f\in L^2(Y)$ the vector $M_h (M_h U_S M_h)^{n-1} (P_X\,f) \in \mathcal{M}$. In addition, $P_X M_h =M_h P_X$. By induction on\footnote{The case $n$=1 is similar to the following.} $n$
\begin{align*}
P_X (M_h U_S M_h)^n P_X\, f
&=M_h P_X U_S \left(M_h \left(M_h U_S M_h\right)^{n-1} P_X\, f\right)\\
&=M_h P_X U_S P_X M_h \left(\left(M_h U_S M_h\right)^{n-1} P_X\, f\right)\\
&=M_h (P_X U_S P_X) P_X M_h \left(\left(M_h U_S M_h\right)^{n-1} P_X\, f\right)\\
&=M_h T_S P_X  M_h \left(\left(M_h U_S M_h\right)^{n-1} P_X\, f\right)\\
&=M_h T_S M_h  P_X \left(\left(M_h U_S M_h\right)^{n-1} P_X\, f\right)\\
&=(M_h T_S M_h)^n\,f
\end{align*}
\end{proof}

\section{Large Deviations}\label{sec:large-deviation}

\subsection{Standing Assumptions}
Let $(Y,\mathcal{Y},m,S)$ be a dynamical system with a state space $\left(X, \mathcal{X}\right)$ having property (M). We denote the associated averaging operator by $T\coloneqq T_S$.

Let $\varphi\colon Y \to \mathbb{R}_+$ be a bounded $\mathcal{X}$-measurable function. 
Set $A_n(y)\coloneqq\frac{1}{n}\sum_{i=0}^{n-1}{\varphi(S^n\,y)}$. 

\subsection{Exponential Moments and the Logarithmic Generating Function}
We would like to bound $m\left(y\in Y \mid A_n(y) \geq \eta \right)$. We follow the standard scheme of proving large deviation bounds, like Cram\'{e}r's theorem, using a bound on the logarithmic generating function, see \cite[\S\S 2.2-2.3]{DemboZeitouni}, \cite[Proof of Lemma 3.2]{KahaleLD}. Notice that we do not assume $T$ to be self-adjoint. 

The approach we take is to bound the exponential moment of the empiric average over the times $0,\ldots,n-1$ by a logarithmic generating function. 
An alternative approach to obtaining a large deviations bound is to use standard geometric moments of degree $n$ for arbitrary large $n$. This is the approach taken in \cite[\S 4]{AFWZ} and \cite[\S 12]{PointsOnSphere}, see also \cite[\S 3.3]{Expanders}.

The Markov inequality with the exponential function implies the following for any real $r>0$.
\begin{align}
m\left(y\in Y \mid A_n(y) \geq \eta\right)
&=m\left(y\in Y \mid \exp{(r n A_n(y))} \geq \exp{(r n \eta)} \right) \nonumber\\
&\leq \int \exp{(r n A_n)} \dif m \exp{(-r n \eta)} \nonumber\\
&=\exp\left\{-n\left(r\eta - \frac{1}{n}\log{\int \exp{(r n A_n)} \dif m} \right)\right\}
\label{eq:markov-exp}
\end{align}
We define the logarithmic generating function and its convex conjugate
\begin{align*} 
&\Lambda_n(r)=\log{\int \exp{(r A_n)} \dif m}\\
&\Lambda_n^*(\eta)=\sup_{r\in\mathbb{R}}{\left(r\eta-\frac{1}{n}\Lambda_n\left(rn\right)\right)}
\end{align*}
the asymptotic  exponential rate in the former bound is given by 
\begin{align*}
&\Lambda(r)=\limsup_{n\to\infty}{\frac{1}{n}\Lambda_n(n r)}\\
&\Lambda^*(\eta)=\sup_{r\in\mathbb{R}}{\left(r\eta-\Lambda\left(r\right)\right)}
\end{align*}

If we bound $\Lambda_n$ we can bound not only $\Lambda^*$ but actually bound effectively the large deviations
probability for a given $n$.
For this we use property (M) and the fact that $\varphi$ is $\mathcal{X}$-measurable.
\begin{align}
\int \exp{(r n A_n)} \dif m&=\int \prod_{i=0}^{n-1} {\left(U^i\, \exp{\left(r\varphi\right)}\right)} \dif m
\nonumber\\
&=\int \exp{(r\varphi)}\cdot \left({U\, M_{\exp{\left(r\varphi\right)}}}\right)^{n-1}1\dif m\nonumber\\
&=\int \exp{(r\varphi/2)}\cdot \left({M_{\exp{\left(r\varphi/2\right)}} \,U\, M_{\exp{\left(r\varphi/2\right)}}}\right)^{n-1}\exp{(r\varphi/2)}
\dif m\nonumber\\
&=\left<\exp{(r\varphi/2)}, 
\left({M_{\exp{\left(r\varphi/2\right)}} \,U\, M_{\exp{\left(r\varphi/2\right)}}}\right)^{n-1}\exp{(r\varphi/2)}\right>
\nonumber\\
&=\left<\exp{(r\varphi/2)}, 
\left({M_{\exp{\left(r\varphi/2\right)}} \,T\, M_{\exp{\left(r\varphi/2\right)}}}\right)^{n-1}\exp{(r\varphi/2)}\right>
\nonumber\\
&\leq \left\|\exp{\left(r\varphi/2\right)}\right\|_2^2 
\left\|{M_{\exp{\left(r\varphi/2\right)}} \,T\, M_{\exp{\left(r\varphi/2\right)}}}\right\|_{L^2(X)}^{n-1}
\label{eq:log-gen-norm}
\end{align}

We are left with bounding $\left\|{M_{\exp{\left(r\varphi/2\right)}} \,T\, M_{\exp{\left(r\varphi/2\right)}}}\right\|_{L^2(X)}$ which we
are going to do using an simple method, an elementary variant of which is essentially due to Kahale \cite[Lemma 3]{Kahale}\footnote{Kahale has computed a large deviations bound for the probability that a random walk on a undirected finite regular graph stays all the time in a subgraph. He has also shown that his bound is optimal in a specific sense.}.

Notice that instead of using the norm $\left\|{M_{\exp{\left(r\varphi/2\right)}} \,T\, M_{\exp{\left(r\varphi/2\right)}}}\right\|_{L^2(X)}$ we could have by Gelfand's formula bound the asymptotic large deviations rate
by the spectral radius
\begin{equation*}
\rho\left({M_{\exp{\left(r\varphi/2\right)}} \,T\, M_{\exp{\left(r\varphi/2\right)}}}\right)_{L^2(X)}
\end{equation*}
Actually if $T$ is compact then as 
a consequence of the Krein-Rutman theorem \cite{KreinRutmanRussian, KreinRutmanEnglish}\footnote{analogue of Perron-Frobenius for compact operators}
 the spectral radius is not only a bound but actually $\Lambda(r)=\rho\left({M_{\exp{\left(r\varphi/2\right)}} \,T\, M_{\exp{\left(r\varphi/2\right)}}}\right)_{L^2(X)}$, at least for some functions $\varphi$.
Furthermore the G\"{a}rtner-Ellis theorem of large deviations implies under some restrictions that the function $\Lambda^*(\eta)$ is not only an upper bound on the large deviations
probability but also a lower bound in a suitable sense.

If $T$ is not only compact but also self-adjoint then ${M_{\exp{\left(r\varphi/2\right)}} \,T\, M_{\exp{\left(r\varphi/2\right)}}}$ is also compact self-adjoint and its
operator norm is equal to the spectral radius.

\subsection{Generalization of Kahale's Method}\label{sec:Kahale}
A large deviation estimate without a self-adjointness assumption has been shown by Ellenberg, Michel and Venkatesh in \cite[Proposition 12.4]{PointsOnSphere}. Our motivation is mainly based upon Linnik's application of large deviation to equidistribution \cite{LinnikBook} and the development on his method by
Ellenberg, Michel and Venkatesh. 

The bound of \cite{PointsOnSphere} is for occupation times  of subsets in a random walk on a directed finite graph. Both our method, which is more general and provides somewhat stronger results, and the method of Ellenberg, Michel and Venkatesh can be traced back to Kahale \cite{Kahale}. The method of \cite{PointsOnSphere} builds upon \cite{AFWZ} while we use a classical approach using exponential moments.

From now on we assume that $0\leq \varphi\leq 1$  $m$-almost everywhere. 
Notice that for any $r\geq 0$
\begin{equation}\label{eq:exp(rphi)-bound}
0\leq e^{r\varphi}=1+\sum_{n=1}^{\infty}\varphi^n\frac{r^n}{n!}
\leq 1+\sum_{n=1}^{\infty}\varphi\frac{r^n}{n!}=1+\varphi(e^r-1)
\end{equation}
The inequalities can be replaced with equalities if $\varphi=\chi_A$ for some $A\in\mathcal{X}$.

For any $f\in L^2(X)$ because $0\leq\varphi\leq 1$ almost anywhere
\begin{align}
\|\exp{(r\varphi/2)} f\|_2&\leq \|\exp{(r/2)} f\|_2=e^{r/2}\|f\|_2 \label{eq:exp(rphi)f-bound}
\end{align}

Assume that $L^2(X)=V\oplus V^\perp$, where $V\subseteq L^2(X)$ is a closed subspace such that both $V$ and $V^\perp$ are
$T$-stable\footnote{As we don't assume $T$ to be self-adjoint stability of $V$ under $T$ does not immediately imply the same for $V^\perp$.}. 
For any $f\in L^2(X)$ denote by $(f)_V$ and $(f)_{V^\perp}$ the orthogonal projections of $f$ to $V$ and $V^\perp$ respectively.

We are going to bound the norm of ${M_{\exp{\left(r\varphi/2\right)}} \,T\, M_{\exp{\left(r\varphi/2\right)}}}$ by its norm on
$V$ and on $V^\perp$. Assume we know that $\|T\|_{V^\perp}\leq\lambda$ for some $1>\lambda>0$ but on $V$ we 
only know $\|T\|_{V}\leq 1$. 

The operator norm can be calculated using
\begin{equation*}
\|M_{\exp{\left(r\varphi/2\right)}} \,T\, M_{\exp{\left(r\varphi/2\right)}}\|_{L^2(X)}=
\sup_{\substack{f,g\in L^2(X) \\ \|f\|_2=\|g\|_2=1} }{\left<g,M_{\exp{\left(r\varphi/2\right)}} \,T\, M_{\exp{\left(r\varphi/2\right)}}\,f\right>}
\end{equation*}

Let $f,g\in L^2(X)$ such that $\|f\|_2=\|g\|_2=1$, then

\begin{align*}
&\left<g,M_{\exp{\left(r\varphi/2\right)}} \,T\, M_{\exp{\left(r\varphi/2\right)}}\,f\right>=
\left<M_{\exp{\left(r\varphi/2\right)}}\,g, T\, M_{\exp{\left(r\varphi/2\right)}}\,f\right>=\\
&\left<e^{r\varphi/2}\,g, T\, e^{r\varphi/2}\,f\right>=\\
&\left<(e^{r\varphi/2}\,g)_{V^\perp}, T\, [(e^{r\varphi/2}\,f)_{V^\perp}]\right>+
\left<(e^{r\varphi/2}\,g)_{V}, T\, [(e^{r\varphi/2}\,f)_{V}]\right>\leq\\
&\lambda\|(e^{r\varphi/2}\,g)_{V^\perp}\|_2 \|(e^{r\varphi/2}\,f)_{V^\perp}\|_2
+\|(e^{r\varphi/2}\,g)_{V}\|_2 \|(e^{r\varphi/2}\,f)_{V}\|_2\leq\\
&\lambda \frac{\|(e^{r\varphi/2}\,g)_{V^\perp}\|_2^2+ \|(e^{r\varphi/2}\,f)_{V^\perp}\|_2^2}{2}
+\|(e^{r\varphi/2}\,g)_{V}\|_2 \|(e^{r\varphi/2}\,f)_{V}\|_2
\end{align*}

The last line follows from the inequality of the means. We can now write  
$\|(e^{r\varphi/2}\,f)_{V^\perp}\|_2^2=\|e^{r\varphi/2}\,f\|_2^2-\|(e^{r\varphi/2}\,f)_{V}\|_2^2$ and
the same with $g$. Using this substitution and applying again the inequality of the means and \eqref{eq:exp(rphi)f-bound} we have
\begin{align}
\big<g&,M_{\exp{\left(r\varphi/2\right)}} \,T\, M_{\exp{\left(r\varphi/2\right)}}\,f\big>
\nonumber\\
&\leq \lambda \frac{\|e^{r\varphi/2}\,g\|_2^2+ \|e^{r\varphi/2}\,f\|_2^2}{2}
+(1-\lambda)\|(e^{r\varphi/2}\,g)_{V}\|_2 \|(e^{r\varphi/2}\,f)_{V}\|_2
\nonumber\\
&\leq \lambda e^r+(1-\lambda)\|(e^{r\varphi/2}\,g)_{V}\|_2 \|(e^{r\varphi/2}\,f)_{V}\|_2
\label{eq:norm-prebound}
\end{align}

\subsection{\texorpdfstring{$L^2_0$}{L^2_0}-Norm}\label{sec:l20-norm-bound}
Take $V=1\mathbb{C}$ to be the subspace of constant functions. If $T$ is self-adjoint then $1-\lambda$ is the spectral gap of
$T$. Furthermore, for any $f\in L^2(X)$ we have by Cauchy-Schwartz 
\begin{equation}\label{eq:specgap-exp(rphi)-bound}
\|(e^{r\varphi/2}\,f)_{V}\|_2=|\int e^{r\varphi/2}\,f \dif m| \leq \|e^{r\varphi/2}\|_2 \|f\|_2=\|e^{r\varphi/2}\|_2
\end{equation}
Let $m(\varphi)=\int \varphi \dif m$, then by \eqref{eq:specgap-exp(rphi)-bound} and \eqref{eq:exp(rphi)-bound}
\begin{equation*}
\|(e^{r\varphi/2}\,g)_{V}\|_2 \|(e^{r\varphi/2}\,f)_{V}\|_2\leq  \|e^{r\varphi/2}\|_2^2
=\int e^{r\varphi}\dif m\leq 1+m(\varphi)(e^r-1)
\end{equation*}

The final bound on the norm implied by \eqref{eq:norm-prebound} and the computation above is
\begin{equation*}
\left\|M_{\exp{\left(r\varphi/2\right)}} \,T\, M_{\exp{\left(r\varphi/2\right)}}\right\|_{L^2(X)}
\leq1+\left[\lambda+(1-\lambda)m(\varphi)\right](e^r-1)
\end{equation*}
Let $\beta=\lambda+(1-\lambda)m(\varphi)$. Notice that $m(\varphi)\leq \beta$ because $0\leq m(\varphi)\leq 1$. Combining 
$\|\exp(r\varphi/2)\|_2^2\leq 1+m(\varphi)(e^r-1)$ and our norm bound with \eqref{eq:exp(rphi)-bound} we have
\begin{align*}
\int \exp{(r n A_n)} \dif m
&\leq[1+m(\varphi)(e^r-1)][1+\beta(e^r-1)]^{n-1}\leq [1+\beta(e^r-1)]^{n}
\end{align*}

By optimizing $r$ we have for $\eta\geq\beta$ that 
\begin{align*}
\Lambda_n^*(\eta)&\geq D(\eta \|\, \beta)\\
D(\eta \,\|\, \beta)&\coloneqq-\eta\log{\frac{\beta}{\eta}}-(1-\eta)\log{\frac{1-\beta}{1-\eta}}
\end{align*}
For all $a,b\in[0,1]$ the function $D(a\,\|\, b)$ is the Kullback-Leibler divergence of two probability distributions on the space of two points, the distributions being $(a,1-a)$ and $(b,1-b)$.

Combining the last bound with \eqref{eq:markov-exp} and \eqref{eq:log-gen-norm} we see that for $\eta\geq\beta$
\begin{equation}\label{eq:ld-simple}
m\left(y\in Y \mid A_n(y) \geq \eta\right)\leq
\exp\left[-n D(\eta \|\,\beta) \right]
\end{equation}

If $\varphi=\chi_A$ for some $A\in\mathcal{X}$ then this bound is exactly the Chernoff-Hoeffding large deviations theorem for i.i.d.\ random variables for occupation of a set of measure $\beta$ instead of $m(A)$. A pleasant feature of our bound is that by taking $\lambda=0$ which is the case for i.i.d.\ random variables
\footnote{The corresponding dynamical system is just the one-sided Bernoulli shift and the state space is the the alphabet of the shift. The map between the two is the projection of a sequence to its first entry.} 
we recover the Chernoff-Heoffding theorem in full strength.

\subsubsection{Norm Exponentiation}
What happens if $\eta>m(\varphi)$ but $\eta\leq\beta\coloneqq \lambda+(1-\lambda)m(\varphi)$? In this case the results of the previous section do not apply verbatim. We follow an established strategy to overcome this difficulty which consists considering the $S^k$ dynamics for high enough $k$, \cite[Proof of Lemma 12.1.13]{PointsOnSphere}.

The bound from the previous section can be improved for $\varphi$ with $m(\varphi)$ small by looking at the $S^k$ dynamics for a fixed $k$. Remember that property (M) implies $T_{S^k}=(T_{S})^k$, hence it has a smaller $L^2_0$ norm, i.e.\ $\|T^k\|_{V^\perp}\leq\lambda^k$.

For $y\in Y$ for which $A_n(y)\geq\eta$ we can for any $1\leq k\leq n$ split the time series $0,1,\ldots,n-1$
into $k$ arithmetic  progressions of length  $>n/k-1$ such that for at least one of them the empirical average of $\phi$ is $\geq\eta$. By applying estimate \eqref{eq:ld-simple} to each such arithmetic progression we have for any
\footnote{If $k>n$ the following bound is trivially correct as in this case $n/k-1<0$.} 
$k\in\mathbb{N}$ such that $\eta\geq \lambda^k+(1-\lambda^k)m(\varphi)$
\begin{align}
&m\left(y\in Y \mid A_n(y) \geq \eta\right)
\leq k\exp\left[-n(\frac{1}{k}-\frac{1}{n}) D(\eta \|\, \lambda^k+(1-\lambda^k)m(\varphi)) \right]
\label{eq:lambda_spec_exp_bound}\\
&=k\exp\left[D(\eta \|\, \lambda^k+(1-\lambda^k)m(\varphi))\right] \exp\left[-n \frac{D(\eta \|\, \lambda^k+(1-\lambda^k)m(\varphi))}{k} \right]
\nonumber\\
&\leq k\exp\left[D(\eta \|\,m(\varphi))\right] \exp\left[-n \frac{D(\eta \|\, \lambda^k+(1-\lambda^k)m(\varphi))}{k} \right]
\nonumber
\end{align}
In the last line we have used that if $\eta\geq \lambda^k+(1-\lambda^k)m(\varphi)\geq m(\varphi)$ then $D(\eta \|\, \lambda^k+(1-\lambda^k)m(\varphi))\leq D(\eta \|\,m(\varphi))$.

\begin{remark}
The multiplicative factor $\exp{D(\eta \|\, m(\varphi))}$ can be optimized a little bit further by noticing that the length of the arithmetic progression is actually $\geq n/k-(k-1)/k$ and keeping the dependence on $k$. As our main interest lies in the exponential rate constant for large $n$, this improvement will be negligible for us and will be unnecessary notationally cumbersome. 
\end{remark}

We now make a comfortable change of variables. We denote $\theta=\lambda^k$, then $k=\log\theta/\log\lambda$ and the condition $k\in\mathbb{N}$ becomes $\theta\in\left\{\lambda^k\mid k\in\mathbb{N}\right\}$. For convenience's sake we use the non-standard notation $\lambda^\mathbb{N}\coloneqq \left\{\lambda^k\mid k\in\mathbb{N}\right\}$. Notice that $\lambda^\mathbb{N}\subseteq (0,\lambda]$.

\begin{thm}\label{thm:LD-specgap}
Let $\phi\colon Y\to[0,1]$ be a $\mathcal{X}$-measurable function. Let $T$ be the averaging operator corresponding to $U_S$ on $X$. Denote $\lambda\coloneqq\|T\|_{L^2_0(X)}\leq 1$. For any $\theta\in\lambda^\mathbb{N}\coloneqq\left\{\lambda^k\mid k\in\mathbb{N}\right\}\subseteq(0,\lambda]$ we have
\begin{enumerate}[label=(\alph*)]
\item If $\eta\geq m(\varphi)$ and $\theta\leq\frac{\eta-m(\varphi)}{1-m(\varphi)}$ then
\begin{align*}
m\left(y\mid \frac{1}{n}\sum_{i=0}^{n-1} \varphi(S^n y)\geq \eta\right)
\leq&\frac{\log\theta}{\log\lambda}\exp{D(\eta \|\, m(\varphi))}\\
&\times\exp\left[-n(-\log\lambda)\frac{D(\eta \|\, \theta+(1-\theta)m(\varphi))}{-\log\theta}\right]
\nonumber
\end{align*}
\item If $\eta\leq m(\varphi)$ and $\theta\leq\frac{m(\varphi)-\eta}{m(\varphi)}$ then
\begin{align*}
m\left(y\mid \frac{1}{n}\sum_{i=0}^{n-1} \varphi(S^n y)\leq \eta\right)
\leq&\frac{\log\theta}{\log\lambda}\exp{D(\eta \|\, m(\varphi))}\\
&\times\exp\left[-n(-\log\lambda)\frac{D(\eta \|\,(1-\theta)m(\varphi))}{-\log\theta}\right]
\nonumber
\end{align*}
\end{enumerate}
\end{thm}
\begin{proof}
Part (a) is \eqref{eq:lambda_spec_exp_bound}.
Part (b) is part (a) applied to the function $1-\varphi$.  
\end{proof}
\section{Walks on \texorpdfstring{$S$}{S}-Algebraic Quotients and Hecke Operators}
In this section we show that some walks on $S$-algebraic quotients fall under the framework of dynamical systems with a state space having property (M). Moreover, the associated averaging operators will be variants of the classical Hecke operators on the modular surface.

The archetypal example is the walks on the $p$-Hecke graph embedded in 
\begin{equation*}
X\coloneqq\lfaktor{\mathbf{PGL}_2(\mathbb{Z})}{\mathbf{PGL}_2(\mathbb{R})}
\end{equation*} 
Set
\begin{equation*}
Y\coloneqq\lfaktor{\mathbf{PGL}_2(\mathbb{Z}[1/p])}{\mathbf{PGL}_2(\mathbb{R})\times \mathbf{PGL}_2(\mathbb{Q}_p)}
\end{equation*}
and equip $Y$ with the right action of the diagonal element $a=\begin{pmatrix}
p & 0 \\ 0 & 1\end{pmatrix}\in\mathbf{PGL}_2(\mathbb{Q}_p)$. Denote $K\coloneqq \mathbf{PGL}_2(\mathbb{Z}_p)$, we recall that there is a natural identification $X\cong\faktor{Y}{K}$. Let $\pi\colon Y\to X$ be the natural projection $y\mapsto yK$.

One can understand the Hecke neighbors of a point $x\in X$ by taking its lift to $Y$ which is the set $xK$, apply $a$ to have $xKa^{-1}$ and project back to $X$ which results in $\pi(xKa^{-1})=x Ka^{-1}K$. The double coset $Ka^{-1}K$ can be written as finite union of size $p+1$ of right $K$ cosets. This gives rise to exactly $p+1$ Hecke neighbors of $x\in X$.

More generally, each point $y\in Y$ projects to a walk on the Hecke graph in $X$ emanating from $\pi(x)\in X$. this walk is given by 
\begin{equation*}
\left(\ldots,\pi(ya^{2}),\pi(ya),\pi(y),\pi(ya^{-1}),\pi(ya^{-2}),\ldots\right)
\end{equation*}

One would like to use the spectral gap results for the $p$-Hecke operator to derive 
a large deviation result in this setting. The flavor of the results we are after is that the measure of points $y\in Y$ whose walks of length $n$ are all contained in a small ball in $X$ decays exponentially with $n$.

Unfortunately, the state space $X$ does not have property (M) for the system $Y$ with the right $a$-action. In particular, the results of Proposition \ref{prop:M-basic-properties} do not hold for the corresponding averaging operator -- the $p$-Hecke operator. For example, $T_{p^2}\neq {(T_p)}^2$.

To amend the situation one replaces $X$ with a finer state space. This is achieved by dividing $Y$ on the right by a smaller compact open sugbroup -- the Iwahori subgroup $I$.

More generally one can consider walks on quotients of $S$-algebraic groups associated to Bruhat-Tits buildings, we construct such walks explicitly and study their relevant properties in \S \ref{sec:walks-buildings}.
\subsection{General Setting}
\subsubsection{Standing Assumptions}\label{sec:standing}
Let $G$ be a second-countable locally compact topological group, such a group is necessarily metrizable with a left invariant proper
\footnote{A metric space is said to be proper if any closed ball is compact.} 
metric \cite{Struble}. Notably, it is $\sigma$-compact.
Fix a lattice
\footnote{Necessarily, $G$ must be unimodular if it has any lattice at all.} 
$\Gamma<G$, set $Y\coloneqq\lfaktor{\Gamma}{G}$. Recall that $\Gamma$ acts freely and properly discontinuously on $G$, in particular, the space $Y$ is locally compact and a left invariant metric on $G$ induces a metric on $Y$.

Denote by $\mathcal{Y}$ the Borel $\sigma$-algebra of $Y$. The space $Y$ carries a unique $G$ invariant probability measure which we denote by $m$. The space $Y$ is not necessarily compact. All $L^q$ spaces will be with respect to this measure $m$.

Let $K<G$ be a compact subgroup. Suppose that $N_G(K)<G$ is an open subgroup of $G$. We define $X\coloneqq\faktor{Y}{K}=\sdfaktor{\Gamma}{G}{K}$ and denote by $\pi\colon Y\to X$ the continuous quotient map. Set $\mathcal{X}$ to be the Borel $\sigma$-algebra of $X$ which we treat as a subalgebra of $\mathcal{Y}$ using the map $\pi$. 

The measure $m$ can be pushed forward to $X$ using $\pi$; by abuse of notation, we denote the pushed forward measure by $m$ as well.
Because of the universal property of the quotient map, there is a natural identification for any $1\leq q \leq\infty$
\begin{equation*}
L^q(X)\cong L^q(Y)^K
\end{equation*}
where the right hand side is the space of $K$ invariant functions.

To be consistent with the right action by inverses, for a subset $F\subseteq G$ for any $y\in Y$ we define $yF\coloneqq \left\{ yf^{-1} \mid f\in F \right\}$.

There is a residual action of the open subgroup $N_G(K)$ on $X$. In particular, for any identity neighborhood $B\subset N_G(K)$ and for all $x\in X$ the set $xB$ is an open neighborhood of $x$. Moroever, these sets form a base for the topology of $X$ around $x$.

Fix $a\in G$. We are interested in the dynamics of the right action by $a$ on $Y$. Obviously, if we want the induced averaging operator on $X$ to have $\|T\|_{L^2_0(X
)}<1$ the space $X$ can not carry an $a$ action, hence $a$ can not belong to $N_G(K)$.

\subsubsection{\texorpdfstring{$S$}{S}-Algebraic Groups}
All our applications will be in the following setting. Let $\mathbf{G}$ be a linear reductive algebraic group defined over $\mathbb{Q}$. Fix $S$ a finite set of places for $\mathbb{Q}$ including at least one finite place. If $\mathbf{G}(\mathbb{R})$ is not compact then $S$ must include the infinite place.

Define $G=G_S\coloneqq\prod_{v\in S} \mathbf{G}(\mathbb{Q}_v)$ and let $\Gamma<G_S$ be a congruence lattice. Let $p\in S$ be a finite place and set $K<\mathbf{G}(\mathbb{Q}_p)$ a compact open subgroup. We have $N_{G_S}(K)\cong\prod_{v\in S\setminus\{p\}} \mathbf{G}(\mathbb{Q}_v)\times N_{\mathbf{G}(\mathbb{Q}_p)}(K)$ which is an open subgroup of $G_S$.

Not any subgroup $K$ will define a state space with property (M). Nevertheless, we shall construct using the action of $\mathbf{G}(\mathbb{Q}_p)$ on the associated affine Bruhat-Tits building natural examples of subgroups $K$ which exhibit property (M) for the right action by $a$.

\subsection{Criteria for Property (M)}
We now discuss what properties of $K$ and $a$ insures that the dynamical system given by the right action of $a$ on $Y$ has property (M) with respect to the state space $X$. 
\begin{defi}
For each $s,t\in\mathbb{Z}$, $t\geq s$, define
\begin{equation*}
K^{(s,t)}=\bigcap_{s\leq i\leq t} a^{-i} K a^i
\end{equation*}
and
\begin{equation*}
K^{(-\infty,t)}\coloneqq\bigcap_{-\infty<i\leq t} a^{-i}K a^i
\end{equation*}

This is a compact subgroup of $G$.
\end{defi}
Notice that if $s'\leq s\leq t\leq t'$ then $K^{(s',t')}\leq K^{(s,t)}$. Moreover, the orthogonal projection $L^2(X)^{K^{(s',t')}}\to L^2(X)^{K^{(s,t)}}$ is given by
\begin{equation*}
f\mapsto f^{K^{(s,t)}}(x)\coloneqq\int_{K^{(s,t)}} f(xk^{-1}) \dif k
\textrm{  for }\mu\textrm{-almost every } x\in X
\end{equation*}
where we integrate with respect to the unique probability Haar measure on a compact subgroup.

Denote by $U\colon L^2(Y)\to L^2(Y)$ the Koopman operator associated to the right action by $a$ on $Y$. A computation shows that $U\left(L^2(X)^{K^{(s,t)}}\right)\subseteq L^2(X)^{K^{(s+1,t+1)}} \subseteq L^2(X)^{K^{(s,t+1)}}$.
\begin{lem}\label{lem:K-diagram}
Suppose that for all $n\in\mathbb{N}$ the following diagram commutes
\begin{center}
	\begin{tikzcd}
		L^2(Y)^{K^{(0,n)}} \arrow{r}{U} \arrow{d} &L^2(Y)^{K^{(0,n+1)}}  \arrow{d}\\
		L^2(Y)^{K^{(0,n-1)}} \arrow{r}{U}  & L^2(Y)^{K^{(0,n)}} \\
	\end{tikzcd}
\end{center}
Then the system $Y$ with the right $a$-action has property (M) with respect to the state space $X$.
\end{lem}
\begin{proof}
First we show that $P_X U P_X\,f=P_X U\,f$ for all $f\in L^2(Y)^{K^{(0,n)}}$ and for all $n\geq 0$. This is proven by induction. The case $n=0$ is trivial. The induction step from $n-1$ to $n$ follows from the complete commutativity of the following diagram.
\begin{center}
	\begin{tikzcd}
		L^2(Y)^{K^{(0,n)}} \arrow{rr}{U} \arrow{d}\arrow[bend right=80]{dd}[left]{P_X} & & L^2(Y)^{K^{(0,n+1)}}  \arrow{d}\arrow[bend left=80]{dd}[right]{P_X}\\
		L^2(Y)^{K^{(0,n-1)}} \arrow{d}{P_X } \arrow{rr}{U}  & & L^2(Y)^{K^{(0,n)}} \arrow{d}{P_X} \\
		L^2(Y)^K \arrow{r}{U}& L^2(Y)^{K^{(0,1)}}\arrow{r}{P_X} & L^2(Y)^K
	\end{tikzcd}
\end{center}
The commutativity of the lower square is the induction assumption and the commutativity of the upper square is the hypothesis of the lemma.

The claim we have just proven implies that $P_X U P_X\,f=P_X U\,f$ for all $f\in \bigoplus_{n=0}^k L^2(Y)^{K^{(0,n)}}$ for all $k\in\mathbb{N}$. Using the continuity of $P_X$ and $U$ we see that $P_X U P_X\,f=P_X U\,f$ for all $f\in \mathcal{M}_a\coloneqq\widehat{\bigoplus}_{n=0}^\infty L^2(Y)^{K^{(0,n)}}$. The subspace $\mathcal{M}_a$ is a closed $U$-stable $L^\infty(X)$-module that satisfies the requirements for property (M).
\end{proof}
\begin{cor}\label{cor:M-criterion}
Assume that there are $\omega_1,\ldots,\omega_k\in K^{(-\infty,0)}$ such that
\begin{equation*}
K=\coprod_{j=1}^k \omega_j K^{(0,1)}
\end{equation*}
Then the system $Y$ with the right $a$-action has property (M) with respect to the state space $X$.
\end{cor}
\begin{proof}
We claim that for any integer $n\geq0$
\begin{equation}\label{eq:cosets-negative}
K^{(-n,0)}=\coprod_{j=1}^k \omega_j K^{(-n,1)}
\end{equation}

We prove this by induction. The case $n=0$ is just the hypothesis of the corollary. We assume the claim for $n-1$ and prove it for $n$
\begin{align*}
K^{(-n,0)}&=a^n K a^{-n} \cap K^{(-(n-1),0)}=a^n K a^{-n} \cap \coprod_{j=1}^k \omega_j K^{(-(n-1),1)}\\
&=\coprod_{j=1}^k \left(a^n K a^{-n} \cap \omega_j K^{(-(n-1),1)}\right)
=\coprod_{j=1}^k \omega_j\left(a^n K a^{-n} \cap  K^{(-(n-1),1)}\right)\\
&=\coprod_{j=1}^k \omega_j K^{(-n,1)}
\end{align*}
In the second line we have used that $\omega_j\in a^n K a^{-n}$ for all $j$.

We conjugate equation \eqref{eq:cosets-negative} by $a^{-(n-1)}$ and have
\begin{equation*}
K^{(0,n-1)}=\coprod_{j=1}^k \left(a^{-(n-1)}\omega_j a^{n-1}\right) K^{(0,n)}
\end{equation*}
We can now check that the diagram of Lemma \ref{lem:K-diagram} does commute. The projection $L^2(Y)^{K^{(0,n)}}\to L^2(Y)^{K^{(0,n-1)}}$ is given for any $n$ by
\begin{equation*}
f(x)\mapsto\frac{1}{k}\sum_{j=1}^k f\left(x a^{-(n-1)}\omega_j a^{n-1}\right)
\end{equation*}
Applying this formula for $n$ and $n+1$ shows that the required diagram commute by a direct computation.
\end{proof}
\begin{remark}[Transferring property (M)]
Notice that the hypothesis of Corollary \ref{cor:M-criterion} depends only on the subgroups $K^{(s,t)}$ for $s\leq t$ integers. If we prove that $K<G$ and $a\in G$ satisfy the conditions of Corollary \ref{cor:M-criterion} then for any other $K'<G'$ and $a'\in G'$ such that $K'^{(s,t)}\cong K^{(s,t)}$ for all $s\leq t$ we have property (M) for the spaces associated to $G'$, $K'$ with the right action by $a'$. 

This allows us to verify property (M) on a complicated space by studying a simpler one.
\end{remark}
\subsection{Walks on Buildings}\label{sec:walks-buildings}
Let now $G$ be a group acting by simplicial automorphisms on a locally finite thick affine building $\Delta$. Denote by $G^\bullet$ the subgroup of type-preserving transformations. Assume that $G^\bullet$ acts strongly transitively on the building $\Delta$, that is the action is transitive on pairs $(\mathcal{C}, \mathcal{A})$ of a chamber $\mathcal{C}$ and an apartment $\mathcal{A}$ containing $\mathcal{C}$.

In this section we construct subgroups of $G$ for which the hypothesis of Corollary \ref{cor:M-criterion} holds with respect to the action by a specific element $a\in G$. This generalizes non-backtracking walks on the Bruhat-Tits tree of $\mathbf{PGL}_2(\mathbb{Q}_p)$.

Fix an apartment $\mathcal{A}\subset\Delta$. Let $N<G$ be the stabilizer of the apartment and denote by $N_\mathrm{trans}$ the elements acting on the apartment $\mathcal{A}$ by translations. Fix $e\neq a\in N_\mathrm{trans}$. We are interested in the action of this $a$.

Fix a special vertex $v_0\in\mathcal{A}$. Let $v_i=a^i.v_0$ for all $i\in\mathbb{Z}$. 
Assume that $v_0$ and $v_1$ \emph{do not} share a common wall. This is a regularity condition on $a$.

Let $\mathcal{S}^+\subset\mathcal{A}$ be a sector in the apartment $\mathcal{A}$ emanating from $v_0$ and including $v_1$.  This sector is unique.
Let $\mathcal{C}^+$ be the chamber in $\mathcal{A}$ which includes $v_0$ and lies at the base of $\mathcal{S}^+$.

Denote by $\mathcal{S}^-$ the unique sector emanating from $v_0$ and including $\mathcal{C}^+$. As $v_0$ and $v_1$ do not share a common wall the sectors $\mathcal{S}^+$ and $\mathcal{S}^-$ have opposite orientations. Let $\mathcal{C}^-$ be the chamber lying at the base of $\mathcal{S}^-$.

Set $I^+<G$ and $I^-<G$ to be the Iwahori subgroups fixing $\mathcal{C}^+$ and $\mathcal{C}^-$ respectively. The subgroup we study is the \emph{arrow} subgroup
\begin{equation*}
K_{\to}=I^+\cap I^-
\end{equation*}

If the  $v_0$ and $v_1$ belong to the same chamber then actually $\mathcal{C}^+=\mathcal{C}^-$ and $K_\to$ is an Iwahori subgroup. 

If $\mathcal{C}^+\neq\mathcal{C}^-$ then the arrow subgroup stabilizes more then just the two chambers $\mathcal{C}^+$, $\mathcal{C}^-$.
\begin{defi}
Let $\mathcal{A}_0$ be an apartment in the building $\Delta$ and let $w_0,w_1\in\mathcal{A}_0$ be two special vertices in the apartment which do not lie on a common wall. Let $\mathcal{S}_0$ be the unique sector in $\mathcal{A}_0$ emanating from $w_0$ and including $w_1$. Let $\mathcal{S}_1$ be the sector emanating from $w_1$ with opposite orientation to $\mathcal{S}_0$. We call the intersection $\mathcal{R}=\mathcal{S}_0\cap\mathcal{S}_1$ a \emph{rhomboid} with corners $w_0$ and $w_1$. It is a bounded convex subset in the apartment $\mathcal{A}_0$ containing $w_0$ and $w_1$. 
\end{defi}

\begin{figure}[h]
	\centering
	\includegraphics[width=\textwidth]{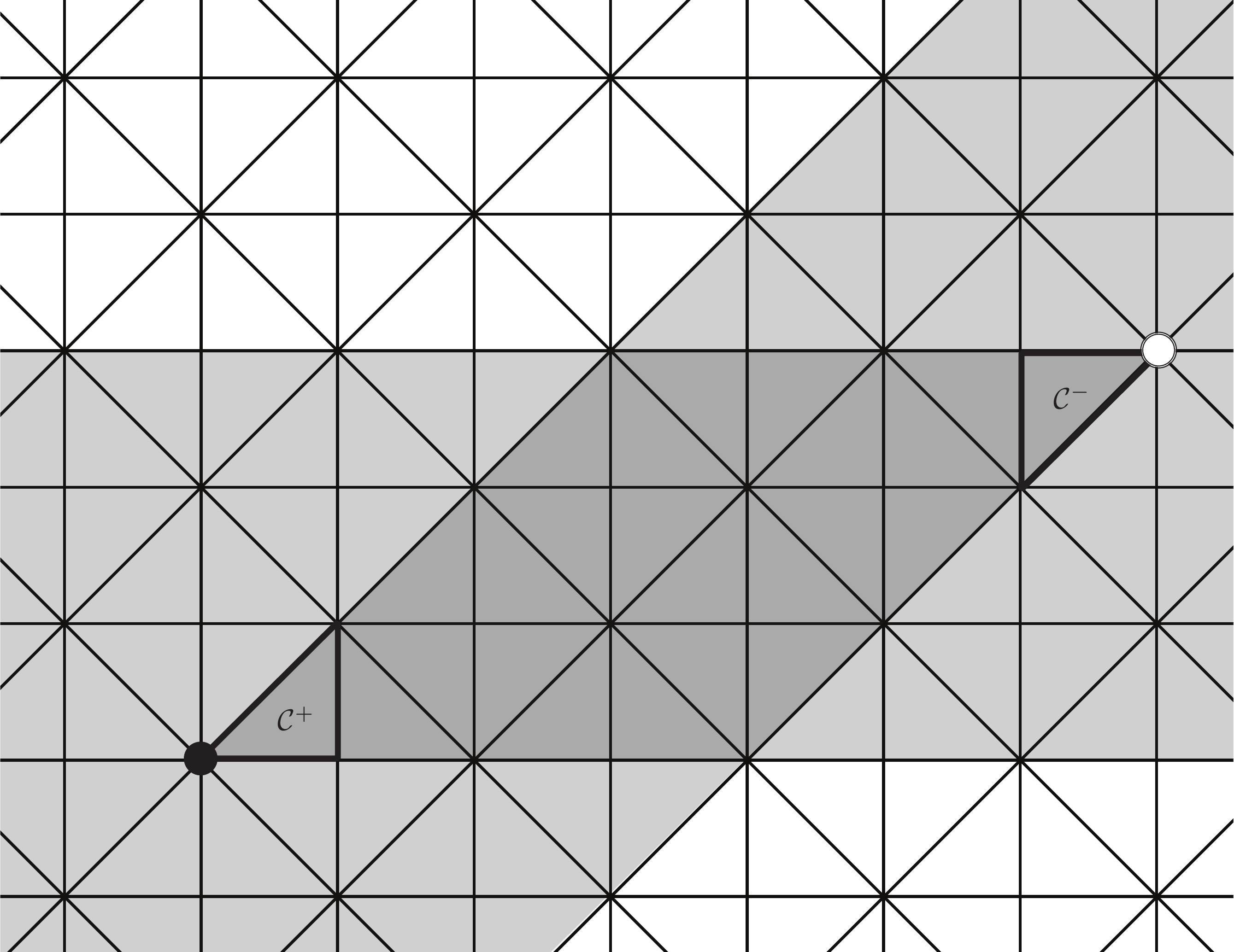}
	\caption{An apartment in a  $\widetilde{C_2}$ building.}
\end{figure}

\begin{prop}
The arrow subgroup $K_\to$ stabilizes pointwise the whole rhomboid $\mathcal{R}$ with corners $v_0$ and $v_1$.
\end{prop}
\begin{proof}
Every minimal gallery $\mathcal{C}_0,\mathcal{C}_1,\ldots,\mathcal{C}_l$ connecting $\mathcal{C}^+$ and $\mathcal{C}^-$ is contained in the apartment $\mathcal{A}$, see \cite[\S 4.5]{GarrettBuldings}. An element $k\in K_\to$ stabilizes pointwise the chambers $\mathcal{C}^+$ and $\mathcal{C}^-$, hence $k.\mathcal{C}_0,k.\mathcal{C}_1,\ldots,k.\mathcal{C}_l$ is also contained in $\mathcal{A}$. Because an apartment is a thin complex and $k$ stabilizes the first and last chamber of the gallery it must stabilize pointwise the whole minimal gallery.

Every chamber $\mathcal{C}_0\subset\mathcal{R}$ belongs to some minimal gallery connecting $\mathcal{C}^+$ and $\mathcal{C}^-$, see \cite[\S 16.1]{GarrettBuldings}, hence it is stabilized by $K_\to$.

\end{proof}

\begin{prop}\label{prop:arrow-M}
The hypothesis of Corollary \ref{cor:M-criterion} holds for the arrow subgroup $K_\to$ and the right action by $a$. Specifically, there exists $\omega_1,\ldots,\omega_k\in K_\to^{(-\infty,0)}$ such that
\begin{equation*}
K_\to=\coprod_{j=1}^k \omega_j K_\to^{(0,1)}
\end{equation*}
\end{prop}
\begin{proof}
Denote by $\mathcal{P}\coloneqq \mathcal{C}^+\cup\mathcal{C}^-$ the pair of chambers defining $K_\to$. 
The group $K_\to^{(0,1)}$ stabilizes pointwise $\mathcal{P}$ and $a.\mathcal{P}$. 

Let $\omega K_\to^{(0,1)}$ be a right  coset in $K_\to$ then
\begin{equation*}
\omega K_\to^{(0,1)}=\left\{g\in G\mid g.\mathcal{P}=\mathcal{P} \textrm{ and } ga.\mathcal{P}=\omega a.\mathcal{P} \right\}
\end{equation*}
Recall that $v_1=a.v_0$. The vertex $v_1$ is stabilized by $\omega K_\to^{(0,1)}$, hence the chamber $\omega a.\mathcal{C}^+$ necessarily has $v_1$ as a vertex. Moreover the metric distance between $\omega a.\mathcal{C}^-$ and $v_1$ is the equal to the distance between $a.\mathcal{C}^-$ and $v_1$ because the vertex $v_1$ is fixed by $\omega$. Hence there is a ball of of fixed finite radius around $v_1$ containing \emph{all} the chamber pairs $\omega a.\mathcal{P}$ coming from all the right cosets of $K_\to^{(0,1)}$  in $K_\to$. By the local finiteness of the building, there are only finitely many possibilities for such pairs of chambers and hence only finitely many right cosets of $K_\to^{(0,1)}$ in $K_\to$.

The chamber pair $a.\mathcal{P}$ is contained in the sector $a.\mathcal{S}^+$ and the chamber pair $\mathcal{P}$ is contained in the sector $\mathcal{S}^-$. Thus for any $\omega\in K_\to$ the pair of chambers $\omega a.\mathcal{P}$ is contained in the sector $\omega a.\mathcal{S}^+$. 

Our plan is to find $\omega'\in G^\bullet$ which stabilizes pointwise $\mathcal{S}^{-}$ and sends $a.\mathcal{S}^+$ to $\omega a.\mathcal{S}^+$ and than show that such $\omega'$ must be contained in $\omega K_\to^{(0,1)}\cap K_\to^{(-\infty,0)}$. This $\omega'$ will be the required representative of the coset and this will finish the proof.

Fix $\omega\in K_\to\setminus K_\to^{(0,1)}$ for the rest of the proof
\footnote{For $\omega\in K_\to^{(0,1)}$ there is nothing to prove as we can take the identity as a representative of $\omega K_\to^{(0,1)}$.}
.

Let $W_\mathrm{aff}$ the affine Weyl group of the building $\Delta$. We shall make use of the $W_\mathrm{aff}$ valued distance function on chambers in $\Delta$, see \cite[\S 15.3]{GarrettBuldings}. For every two chamber $\mathcal{C}, \mathcal{D}\in\Delta$ we have $\delta(\mathcal{C}, \mathcal{D})\in W_\mathrm{aff}$. If $\mathcal{A}_0$ is any apartment containing $\mathcal{C}$ and $\mathcal{D}$ we have a canonical action of $W_\mathrm{aff}$ on $\mathcal{A}_0$ and
\footnote{To be consistent with the definitions of \cite[\S 15.3]{GarrettBuldings} we need to let $\delta(\mathcal{C},\mathcal{D})$ act on $\mathcal{A}_0$ on the right after fixing an identification of $\mathcal{A}_0$ with $W_\mathrm{aff}$.}
$\delta(\mathcal{C},\mathcal{D}).\mathcal{D}=\mathcal{C}$. Moreover, for any $g\in G$ and any two chambers $\mathcal{C}, \mathcal{D}$ we have $\delta(g.\mathcal{C}, g.\mathcal{D})=\delta(\mathcal{C},\mathcal{D})$.

Let $\overline{w}\in W_\mathrm{sph}$ be the longest element of the spherical Weyl group.
We immediately see that 
\begin{equation}\label{eq:delta-omega}
\delta(\mathcal{C}^-,\omega a.\mathcal{C}^+)=\delta(\omega.\mathcal{C}^-,\omega a.\mathcal{C}^+)=\delta(\mathcal{C}^-,a.\mathcal{C}^+)=\overline{w}
\end{equation}
We wish to understand $\omega a.\mathcal{S}^+\cap \mathcal{S}^-$. Obviously, the vertex $v_1$ belongs to the intersection, we claim that this all the intersection.
The set $\omega a.\mathcal{S}^+\cap \mathcal{S}^-$ lies in the intersection of the two apartments $\mathcal{A}\cap \omega.\mathcal{A}$. The intersection of the two apartments is a convex subset in $\mathcal{A}$ including $v_1$.   Hence if $\omega a.\mathcal{S}^+\cap \mathcal{S}^-$ is strictly larger then $\{v_1\}$ then by convexity the chambers $\omega a.\mathcal{C}^+$ and $\mathcal{C}^-$ share a common face of dimension greater then $1$, this contradicts the fact $\delta(\mathcal{C}^-,\omega a.\mathcal{C}^+)=\overline{w}$.
We have thus proven $\omega a.\mathcal{S}^+\cap \mathcal{S}^-=\{v_1\}$.

We now construct $\omega'\in G$ such that $\omega'.\mathcal{S}^-=\mathcal{S}^-$ and $\omega'a.\mathcal{S}^+=\omega a.\mathcal{S}^+$. Look at the subset $\mathcal{D}_\omega\coloneqq\mathcal{S}^-\cup \omega a.\mathcal{S}^+$. We claim that this subset is \emph{strongly isometric} to $\mathcal{D}\coloneqq\mathcal{S}^-\cup  a.\mathcal{S}^+\subset\mathcal{A}$. Recall that two subsets of chambers, $\mathcal{E}$ and $\mathcal{F}$, are strongly isometric if there exists a bijection $f\colon\mathcal{E}\to\mathcal{F}$
such that for any two chambers $\mathcal{C},\mathcal{D}\in\mathcal{E}$ we have $\delta(f(\mathcal{C}), f(\mathcal{D}))=\delta(\mathcal{C}, \mathcal{D})$.

We define the strong isometry $f\colon\mathcal{D}\to\mathcal{D}_\omega$ by 
\begin{equation*}
f(\mathcal{C})=
\begin{cases}
\mathcal{C} & \textrm{if }\mathcal{C}\in\mathcal{S}^- \\
\omega.\mathcal{C} & \textrm{if }\mathcal{C}\in a.\mathcal{S}^+ \\
\end{cases}
\end{equation*}
This is well defined because the intersection of $\mathcal{S}^-$ and $\omega a.\mathcal{S}^+$ includes only the vertex $v_1$. It is obvious that $f$ preserves $\delta$ between two chamber which are both contained either in $\mathcal{S}^-$ or in $a.\mathcal{S}^+$. Let now $\mathcal{C}\in\mathcal{S}^-$ and $\mathcal{C}'\in a.\mathcal{S}^+$. We use the multiplicative property of $\delta$ and have
\begin{align*}
\delta(f(\mathcal{C}),f(\mathcal{C}'))&=\delta(\mathcal{C},\omega.\mathcal{C}')
=\delta(\mathcal{C},\mathcal{C}^-)\delta(\mathcal{C}^-,\omega a.\mathcal{C}^+)
\delta(\omega a.\mathcal{C}^+,\omega.\mathcal{C}')\\
&=\delta(\mathcal{C},\mathcal{C}^-)\delta(\mathcal{C}^-,a.\mathcal{C}^+)
\delta(a.\mathcal{C}^+,\mathcal{C}')\\
&=\delta(\mathcal{C},\mathcal{C}')
\end{align*}
One can show using an analogues calculation that $\delta(f(\mathcal{C}'),f(\mathcal{C}))=\delta(\mathcal{C}',\mathcal{C})$.

We have shown that $\mathcal{D}_\omega$ is strongly isometric to a subset of an apartment, hence there exists an apartment $\mathcal{A}_\omega$ such that $\mathcal{D}_\omega\subset\mathcal{A}_\omega$, see \cite[\S 15.5]{GarrettBuldings}.
Notice that $\mathcal{C}^-\subset\mathcal{S}^-\subset\mathcal{A}_\omega$.
We now use the strong transitivity of $G^\bullet$ acting on $\Delta$ to deduce that there is $\omega'\in G^\bullet$ sending the pair $\left(\mathcal{C}^-,\mathcal{A}\right)$ to $\left(\mathcal{C}^-,\mathcal{A}_\omega\right)$. 

This $\omega'$ must send $\mathcal{S}^-$ to the sector emanating from $\omega'(v_1)=v_1$ in the orientation defined by $\omega'(\mathcal{C}^-)=\mathcal{C}^-$ which is exactly $\mathcal{S}^-$. In particular, as sectors are thin and $\omega'$ preserves types we conclude that $\omega'$ fixes $\mathcal{S}^-$ pointwise.

Moreover, we have 
\begin{equation*}
\delta(\mathcal{C}^-,\omega'a.\mathcal{C}^+)
=\delta(\omega'.\mathcal{C}^-,\omega'a.\mathcal{C}^+)
=\delta(\mathcal{C}^-,a.\mathcal{C}^+)=\overline{w}
\end{equation*}
Also $\delta(\mathcal{C}^-,\omega a.\mathcal{C}^+)=\overline{w}$, but both $\omega a.\mathcal{C}^+$ and $\omega' a.\mathcal{C}^+$ belong to the apartment $\mathcal{A}_w$ and have the same $\delta$ distance from $\mathcal{C}^-\in\mathcal{A}_w$ hence $\omega' a.\mathcal{C}^+=\omega a.\mathcal{C}^+$. We see that $\omega'.a\mathcal{S}^+$ is a sector in the same apartment $\mathcal{A}_\omega$ as $\omega.a\mathcal{S}^+$ with the same base point $v_1$ and the same chamber at base, hence $\omega'.a\mathcal{S}^+=\omega.a\mathcal{S}^+$ as required.

We are left only with proving that $\omega'\in\omega K_\to^{(0,1)}\cap K_\to^{(-\infty,0)}$. Because $\omega'$ stabilizes pointwise $\mathcal{S}^-$, to prove that $\omega'\in K_\to^{(-\infty,0)}$ it is enough to show that $a^{-n}.\mathcal{P}\subset \mathcal{S}^{-}$ for all $n\geq0$. We prove this claim by induction. For $n=0$ it follows from the definition of $\mathcal{S}^-$. Assume the claim for $n$: $a^{-n}.\mathcal{P}\subset\mathcal{S}^{-1}$.  By applying $a^{-1}$ we have $a^{-(n+1)}.\mathcal{P}\subset a^{-1}.\mathcal{S}^{-1}$. The sector $a^{-1}.\mathcal{S}^{-1}$ emanates from $v_0\in\mathcal{S}^{-}$ and because $a^{-1}$ acts as a translation on the apartment $\mathcal{A}$ we have that $a^{-1}.\mathcal{S}^{-1}$ has the same orientation in $\mathcal{A}$ as $\mathcal{S}^{-}$. Hence $a^{-1}.\mathcal{S}^{-} \subset \mathcal{S}^-$ and $a^{-(n+1)}.\mathcal{P}\subset \mathcal{S}^{-1}$ as required.

At last we show that $\omega'.\mathcal{P}=\omega.\mathcal{P}$ which would imply $\omega' K_\to^{(0,1)}=\omega K_\to^{(0,1)}$. We already know that $\omega' a.\mathcal{C}^+=\omega a.\mathcal{C}^+$, where the equality must be pointwise as both $\omega'$ and $\omega$ are type-preserving
\footnote{The element $\omega$ belongs to an Iwahori hence preserves types.}
. Now both $\omega'a.\mathcal{C}^-$ and $\omega a.\mathcal{C}^-$ belong the the apartment $\mathcal{A}_\omega$ and $\delta(\omega a.\mathcal{C}^+,\omega a.\mathcal{C}^-)=\delta(a.\mathcal{C}^+, a.\mathcal{C}^-)=\delta(\omega a.\mathcal{C}^+,\omega' a.\mathcal{C}^-)$, so $\omega' a.\mathcal{C}^-=\omega a.\mathcal{C}^-$ with pointwise equality.
\end{proof}
\subsection{Walks on \texorpdfstring{$S$}{S}-Algebraic Quotients}
\label{sec:walks-s-algebraic}
In this section we finally present the main class of systems with property (M) which we are interested in. Let $\mathbf{G}$ be a reductive linear algebraic group defined over $\mathbb{Q}$ and let $S$ be a finite set of places of $\mathbb{Q}$. Set $G_S\coloneqq\prod_{v\in S} \mathbf{G}(\mathbb{Q}_v)$.

We fix some finite place $p\in S$ and let $G_p\coloneqq \mathbf{G}(\mathbb{Q}_p)$ act on the associated affine Bruhat-Tits building $\Delta$. The building $\Delta$ has been constructed for $\mathbf{SL}_n$ by Goldman and Iwahori \cite{GoldmanIwahori}, for split groups by Iwahori and Matsumoto \cite{IwahoriMatsumoto} and generally by Bruhat and Tits \cite{BruhatTits}.

\begin{prop}\label{prop:M-algebraic}
Let $\Gamma<G_S$ be some lattice and define $Y\coloneqq \lfaktor{\Gamma}{G_S}$. Let $a\in G_p$ be a semisimple element. Denote by $A< G_p$ the maximal split subtorus
\footnote{The torus $A$ is not necessarily a maximal torus in $G_p$ as we have not assumed that $\mathbf{G}$ is split over $\mathbb{Q}_p$.} 
of a maximal torus containing $a$.

The element $a^{-1}$ acts on the apartment in $\Delta$ associated to $A$, assume that it sends a fixed special vertex $v_0$ to a different special vertex $v_1$ which does not have a shared wall with $v_0$. This assumption is fulfilled whenever $a$ is $\mathbb{Q}_p$-regular and does not belong to a compact subgroup.

Let $K_\to<G_p$ be the arrow subgroup corresponding to $a^{-1}$ and the vertex $v_0$ in the apartment associated to $A$. Define $X\coloneqq \faktor{Y}{K_\to}$. Then $Y\to X$ has property (M) with respect to the right action by $a$.
\end{prop}
\begin{proof}
The group $K_\to$, being the intersection of two Iwahori subgroup, is a compact open subgroup of $G_p$. In particular, it is a compact subgroup of $G_S$ and $X$ is a Hausdorff quotient space.

The claim is a direct consequence of Corollary \ref{cor:M-criterion} and Proposition \ref{prop:arrow-M}.
\end{proof}

\section{Applications to Equidistribution}
We return to the general setting of \S \ref{sec:standing}. We study the dynamical system $\left(Y, \mathcal{Y}, m, S \right)$, 
$Y\coloneqq\lfaktor{\Gamma}{G}$ where $G$ is a second-countable locally compact topological group, $\Gamma<G$ -- a lattice, $m$ -- the probability Haar measure on $Y$ and the transformation $S$ is the right action by some fixed $a\in G$. The state space is $X\coloneqq \faktor{Y}{K}$ where $K<G$ is a compact subgroup such that $N_G(K)<G$ is an open subgroup.

Moreover we denote the associated averaging operator by $T$ and assume that $\lambda\coloneqq\|T\|_{L^2_0(X)}<1$.

For any $a$-invariant Borel probability measure $\nu$ on $Y$ we denote by $h_\nu(a)$ the measure theoretic entropy of $\nu$ with respect to the right $a$-action. In particular, $h_m(a)$ is the entropy of the Haar measure.
\subsection{Bowen Balls}
\begin{defi}
Fix $a\in G$.
Let $B\subseteq G$ be an identity neighborhood. For $s,t\in\mathbb{Z}$ such that $s\leq t$ we define the $B^{(s,t)}$ Bowen ball to be
\begin{equation*}
B^{(s,t)}\coloneqq \bigcap_{s\leq i\leq t} a^{-i} B a^i
\end{equation*}

We say that $B$ is $(a,h)$-homogeneous for some $h\geq0$ if there exists $C_{B,h}>0$ which depends only on $h$, $B$, $G$ and $a$ such that for any $s\leq0\leq t$
\begin{equation*}
m(B^{(s,t)})\geq C_{B,h}\exp(-h (t-s+1))
\end{equation*}
\end{defi}

\begin{defi}
The Haar measure $m$ is $a$-homogeneous if for all $h>h_m(a)$ there is a base for the topology of $G$ around the identity of $(a,h)$-homogeneous neighborhoods.
\end{defi}

\begin{remark}
In important cases much stronger homogeneity properties for identity neighborhoods hold with respect to the Haar measure.
Specifically, we might have a concrete value for $C_{B,h}$ in terms of $m(B)$ and perhaps even be able to set $h=h_m(a)$.

A stronger homogeneity assumption will produce a stronger effective equidistribution result in Theorem \ref{thm:effective}.

In Proposition \ref{prop:homogeneous} we show that in the $S$-arithmetic setting the Haar measure is $(a,h_m(a))$-homogeneous.
\end{remark}

\begin{defi}
Let $\varphi\colon Y\to[0,1]$ be $\mathcal{X}$-measurable. For any identity neighborhood $B\subset N_G(K)$ we define
\begin{align*}
&\varphi^B(y)=\esssup_{b\in B} \varphi(y b)=\esssup \varphi(y B^{-1})\\
&\varphi_B(y)=\essinf_{b\in B} \varphi(y b)=\esssup \varphi(y B^{-1})
\end{align*}

The assumption $B\subset N_G(K)$ implies that the functions $\varphi^B(y)$ and $\varphi_B(y)$ are $\mathcal{X}$-measurable.
\end{defi}

\begin{remark}
If $G$ is an $S$-algebraic group with a compact part at infinity then one can look at neighborhoods $B\subseteq K_0$ where $K_0$ is a compact-open subgroup of $G$. In this case if $\varphi$ is $K_0$ invariant, i.e.\ $K_0$-smooth, then $\varphi^B=\varphi_B=\varphi$.
\end{remark}

\begin{lem}\label{lem:semicont-bound}
Let $\varphi\colon Y\to[0,1]$ be $\mathcal{X}$-measurable and $B\subset N_G(K)$  an identity neighborhood.
Let $s,t\in \mathbb{Z}$ such that $s\leq t$. For all $y\in Y$ and $z\in y B^{(s,t)}$
\begin{enumerate}[label=(\alph*)]
\item if $\varphi$ is lower-semicontinuous then
\begin{equation*}
\sum_{i=s}^{t} \varphi^{B}(z a^{-i})\geq \sum_{i=s}^{t} \varphi(y a^{-i})
\end{equation*}\item if $\varphi$ is upper-semicontinuous then
\begin{equation*}
\sum_{i=s}^{t} \varphi_{B}(z a^{-i})\leq \sum_{i=s}^{t} \varphi(y a^{-i})
\end{equation*}
\end{enumerate}
\end{lem}
\begin{remark}
If $\varphi$ is the characteristic function of an open set then it is lower-semicontinuous and if it is the characteristic function of a closed set it is upper-semicontinuous.
\end{remark}
\begin{proof}
If $z\in y B^{(s,t)}$ then for all $s\leq i\leq t$ we have $z a^{-i}\in y a^{-i} B$, i.e.\ there exists $b\in B$ such that $z a^{-i} = y a^{-i} b^{-1} \Leftrightarrow z a^{-i} b = y a^{-i} $. This implies
\begin{align*}
&\varphi^B(z a^{-i})=\esssup_{b\in B} \varphi(z a^{-i} b)\geq\phi(y a^{-i})\\
&\varphi_B(z a^{-i})=\essinf_{b\in B} \varphi(z a^{-i} b)\leq\phi(y a^{-i})
\end{align*}
In the inequalities on the right we have used lower-semicontinuity or upper-semicontinuity respectively.
\end{proof}

\begin{lem}\label{lem:disjoint-bowen}
Let $B\subset G$ be an identity neighborhood. If $B_{0}\subset G$ is an identity neighborhood such that ${B_{0}}^{-1}B_{0}\subseteq B$ then for all $s,t\in\mathbb{Z}$ such that $s\leq t$
\begin{equation*}
{B_0^{(s,t)}}^{-1}B_0^{(s,t)}\subseteq B^{(s,t)}
\end{equation*}
\end{lem}
\begin{remark}
If $B$ is a subgroup then $B_0=B$ satisfies the conditions of the lemma.
\end{remark}
\begin{proof}
For any $s\leq i \leq t$
\begin{equation*}
\left(a^{-i} B_0 a^i\right)^{-1}\left(a^{-i} B_0 a^i\right)=a^{-i} {B_0}^{-1}B_0 a^i\subseteq
a^{-i} B a^i
\end{equation*}
hence
\begin{align*}
{B_0^{(s,t)}}^{-1}B_0^{(s,t)}&=\bigcap_{s\leq i \leq t}\left(a^{-i} B_0 a^i\right)^{-1}\cdot
\bigcap_{s\leq j \leq t}\left(a^{-j} B_0 a^j\right)\\
&\subseteq \bigcap_{s\leq i \leq t}\left(a^{-i} B_0 a^i\right)^{-1} \left(a^{-i} B_0 a^i\right)\\
&\subseteq \bigcap_{s\leq i \leq t} \left(a^{-i} B a^i\right)=B^{(s,t)}
\end{align*}
\end{proof}
\subsection{Effective Equidistribution}
The idea to use large deviation estimates to prove equidistribution of measures with good separation properties goes back to Linnik \cite{LinnikBook}, who has already proved an effective version of his result. It has been developed further by Ellenberg, Michel and Venkatesh \cite{PointsOnSphere}. 
The novelty of our result is that it applies to measures whose separation is sub-optimal.

The relevant separation property for the large deviations methods is similar to the separation property used in showing that the limit measure has positive entropy \cite[Proposition 3.3]{ELMVPeriodic}. We require that the average over a large compact set $\Omega$  of the measure of Bowen balls on a \emph{single} scale $n_0$ is exponentially small. 
\begin{equation}\label{eq:bowen-ave}
-\frac{1}{2n_0}\log\int_\Omega \mu\left(yB^{(-n_0,n_0)}\right)\dif\mu(y)\geq h
\end{equation}
This should be contrasted with entropy, which by the Brin-Katok theorem \cite{BrinKatok}, implies that the average over a large set of the logarithm of the measure of Bowen balls is  small for all large enough scales, $\forall n\geq N$. 
\begin{equation*}
-\frac{1}{2n}\int_\Omega \log \mu\left(yB^{(-n,n)}\right)\dif\mu(y)\geq h_m(a)-\varepsilon
\end{equation*}

Obviously, this two properties are related but are not interchangeable. Entropy does not provide a single scale on which the average size of the Bowen balls is small, it rather provides a scale on which the average size of the logarithm of the Bowen ball is small. For a fixed Bowen Ball $B^{(-n_0,n_0)}$ this is a weaker property then \eqref{eq:bowen-ave} because $(-\log)$ is a strictly convex function.

On the other hand, entropy requires information regarding size of Bowen balls on all large enough scales, while the separation property we provide is on a single scale. Specifically, the measure $\mu$ might have zero entropy and our theorem can still provide non-trivial information. This is extremely useful in applications like \cite{LinnikBook} and \cite{PointsOnSphere}.

\begin{thm}\label{thm:effective}
Fix $a\in G$ and assume that $Y\to X$ has property (M) with respect to the right action by $a$. Fix further $B\subset G$ an identity neighborhood and a compact subset $\Omega\subset Y$.
Let $\mu$ be an $a$-invariant Borel probability measure on $Y$. 
Fix two distinct integers $s\leq0\leq t$  and let $n=t-s+1$.
Define
\begin{equation*}
\alpha\coloneqq-\frac{1}{n}\log\int_\Omega \mu\left(yB^{(s,t)}\right)\dif\mu(y)
\end{equation*} 

Fix $h>0$ and let $B_0\subset N_G(K)$ be an identity neighborhood such that
\begin{enumerate}
\item $B_0$ is $(a,h)$-homogeneous,
\item ${B_0}^{-1}B_0\subseteq B$,
\item the projection $\pi\colon G\to Y$ is injective when restricted to $yB_0$ for any\footnote{Such $B_0$ necessarily exists by the continuity and non vanishing of the injectivity radius.} $y\in \Omega$.
\end{enumerate}

For a $\mathcal{X}$-measurable function $\varphi\colon Y\to[0,1]$
\begin{enumerate}[label=(\alph*)]
\item if $\varphi$ is lower-semicontinuous and $1>\mu(\varphi)\geq m(\varphi^{B_0})$ then for any 
\begin{equation*}
\mu(\Omega)>\kappa \geq \frac{1-\mu(\varphi)}{1- m(\varphi^{B_0})};\quad
\lambda^\mathbb{N}\ni\theta\leq\frac{\left[{\mu(\varphi)-(1-\kappa)}\right]/{\kappa}- m(\varphi^{B_0})}{1-m(\varphi^{B_0})}
\end{equation*}
we have
\begin{align*}
&\frac{D\left(\left[{\mu(\varphi)-(1-\kappa)}\right]/{\kappa} \,\|\, \theta+(1-\theta)m(\varphi^{B_0})\right)}{-\log\theta}
\leq\frac{h-\alpha}{-\log\lambda}\\
&+\frac{1}{n}\Big[\log\left(\frac{\log\theta}{\log\lambda}\right) + D\left(\mu(\varphi)\,\|\,m(\varphi^{B_0})\right)\\
&-2\log\frac{\mu(\Omega)-\kappa}{2}
-\log C_{B_0,h}\Big]
\end{align*}
\item if $\varphi$ is upper-semicontinuous and $0<\mu(\varphi)\leq m(\varphi_{B_0})$ then for any 
\begin{equation*}
\mu(\Omega)>\kappa \geq \frac{\mu(\varphi)}{m(\varphi_{B_0})};\quad
\lambda^\mathbb{N}\ni\theta\leq\frac{m(\varphi_{B_0})-\mu(\varphi)/{\kappa}}{m(\varphi_{B_0})}
\end{equation*}
we have
\begin{align*}
&\frac{D\left(\mu(\varphi)/\kappa \,\|\, (1-\theta)m(\varphi^{B_0})\right)}{-\log\theta}
\leq\frac{h-\alpha}{-\log\lambda}\\
&+\frac{1}{n}\Big[\log\left(\frac{\log\theta}{\log\lambda}\right) + D\left(\mu(\varphi)\,\|\,m(\varphi^{B_0})\right)\\
&-2\log\frac{\mu(\Omega)-\kappa}{2}
-\log C_{B_0,h}\Big]
\end{align*}
\end{enumerate}
\end{thm}
\begin{remark}
In the favorable case that $Y$ is compact one can take $\Omega=Y$ and deduce a simpler statement. Nevertheless, the statement for a general $\Omega$ is useful for us even in compact spaces. This is because often it is easy to find a big compact set all whose points have some good typical behavior with respect to $\mu$ but that is not the case for the whole space $Y$. 
\end{remark}
\begin{proof}
The idea of the proof is to use the bound on the average size of $B$-Bowen balls to generate many mutually disjoint $B_0$-Bowen balls around points which have typical statistics with respect to $\mu$. Then use the fact that $B_0$ is an $a$-homogeneous neighborhood to show a lower bound on the mass of the union of these disjoint $B_0$-Bowen balls with respect to the \emph{Haar} measure $m$. On the other hand, all the points in these $B_0$-Bowen balls behave typically with respect to $\mu$, at least on the scale of $B_0$. This allows us to derive an upper bound on the Haar measure of the union using our large deviation results in terms of $m(\varphi^{B_0})$ or $m(\varphi_{B_0})$.

Let $\kappa_B>0$ be a constant to be optimized later and let $\kappa>0$ be a free variable whose range will be set later as well.
For a function $f\colon Y\to\mathbb{R}$ denote for any $y\in Y$
\begin{equation*}
A^{(s,t)}f\,(y)\coloneqq\frac{1}{t-s+1}\sum_{i=s}^t f(y a^{-i})
\end{equation*}
Because $m$ is $a$ invariant we have $\int A^{(s,t)}f\dif\mu=\int f\dif\mu$ for any integrable $f$.

\paragraph{Points with good statisics and small Bowen balls.}
We want first to find $\mu$-many points whose $(s,t)$ orbit has statistics similar to $\mu$.
Define
\begin{align*}
&Y^\varphi\coloneqq\left\{y\in Y \mid A^{(s,t)}\varphi\,(y)>  \frac{\mu(\varphi)-(1-\kappa)}{\kappa} \right\}\\
&Y_\varphi\coloneqq\left\{y\in Y \mid A^{(s,t)}\varphi\,(y)<  \frac{\mu(\varphi)}{\kappa} \right\}
\end{align*}
Then by Markov's inequality
\begin{equation*}
\mu(Y^\varphi),\mu(Y_\varphi)\geq 1-\kappa
\end{equation*}
Next we define $\Omega_\varphi\coloneqq Y_\varphi\cap \Omega$ and $\Omega^\varphi\coloneqq Y^\varphi\cap \Omega$. Obviously,
\begin{equation*}
\mu(\Omega^\varphi),\mu(\Omega_\varphi)\geq 1-\kappa-\mu(\Omega^{\mathrm{C}})=\mu(\Omega)-\kappa
\end{equation*}
Necessarily, we need $\kappa<\mu(\Omega)$ for this expression to be useful. 

Next we look for $\mu$-many points for which the measure of the $B^{(s,t)}$ Bowen ball around them is close to $\exp(-\alpha n)$. Define
\begin{equation*}
\Omega_B\coloneqq\left\{y\in Y\mid \mu\left(y B^{(s,t)}\right)\leq \exp(-\alpha n)/\kappa_B \right\}
\end{equation*}
Again by Markov's inequality
\begin{equation*}
\mu(\Omega_B)\geq 1-\kappa_B
\end{equation*}

Finally, we have a bound on the $\mu$-measure of the sets with good statistics and small Bowen balls.
\begin{equation*}
\mu(\Omega^\varphi\cap\Omega_B), \mu(\Omega_\varphi\cap\Omega_B)\geq
\mu(\Omega)-\kappa-\kappa_B
\end{equation*}

\paragraph{Disjoint Bowen balls with good statistics.} 
We present an inductive procedure that constructs finite sets of points $\Sigma^\varphi,\Sigma_\varphi$ such that
\begin{enumerate}[label=(\alph*)]
\item $\Sigma^\varphi\subseteq\Omega^\varphi$ and $\Sigma_\varphi\subseteq\Omega_\varphi$, 
\item for any two points $y_1,y_2\in\Sigma^\varphi$ or $y_1,y_2\in\Sigma_\varphi$ we have $y_1 B_0^{(s,t)}\cap y_2 B_0^{(s,t)}=\emptyset$,
\item $|\Sigma^\varphi|,|\Sigma_\varphi|\geq \kappa_B\frac{\mu(\Omega)-\kappa-\kappa_B}{\exp(-\alpha n)}$.
\end{enumerate} 

Let $\Xi_0=\Omega^\varphi\cap\Omega_B$ or $\Xi_0=\Omega_\varphi\cap\Omega_B$. The inductive procedure goes as follows. Notice that $\mu(\Xi_0)\geq \mu(\Omega)-\kappa-\kappa_B$.
\begin{enumerate}
\item \emph{Inductive assumption:} we have chosen $y_1,\ldots,y_k$ in $\Omega^\varphi\cap\Omega_B$ or  $\Omega_\varphi\cap\Omega_B$ respectively such that $y_i \not\in \bigcup_{j=1}^{i-1} y_j B^{(s,t)}$ for any $1\leq i \leq k$. Moreover, there is a set $\Xi_k\subseteq\Omega^\varphi\cap\Omega_B$ or $\Xi_k\subseteq\Omega_\varphi\cap\Omega_B$ respectively such that $\mu(\Xi_k)\geq \mu(\Omega)-\kappa-\kappa_B-k\exp(-\alpha n)/\kappa_B$ and for each $1\leq i\leq k$ we have $y_i B^{(s,t)}\cap\Xi_k=\emptyset$.
\item If $\mu(\Xi_k)=0$ finish the procedure.
\item If $\mu(\Xi_k)>0$ choose $y_{k+1}\in \Xi_k$ and set $\Xi_{k+1}\coloneqq\Xi_k\setminus y_{k+1}B^{(s,t)}$. 
\item Obviously, $y_{k+1} B^{(s,t)}\cap\Xi_{k+1}=\emptyset$. We have $\Xi_{k+1}\subseteq\Xi_k$ hence§ for all $1\leq i\leq k$ also $y_i B^{(s,t)}\cap\Xi_{k+1}=\emptyset$.
\item Notice that $y_{k+1}\in\Xi_k$ and $\Xi_k\cap\bigcup_{i=1}^k y_i B^{(s,t)}=\emptyset$, thus $y_{k+1}\not\in \bigcup_{i=1}^{k} y_i B^{(s,t)}$.
\item $\mu(\Xi_{k+1})\geq\mu(\Xi_k)-\mu\left(y_{k+1} B^{(s,t)}\right)\geq\mu(\Xi_k)-\exp(-\alpha n)/\kappa_B\geq \mu(\Omega)-\kappa-\kappa_B-(k+1)\exp(-\alpha n)/\kappa_B$, where we have used the fact that $y_{k+1}\in\Omega_B$.
\item Items (4), (5) and (6) show that our choice of $\Xi_{k+1}$, $y_{k+1}$ fulfill the inductive assumption.
\end{enumerate}

The procedure terminates when $0=\mu(\Xi_k)\geq \mu(\Omega)-\kappa-\kappa_B-k\exp(-\alpha n)/\kappa_B$, meaning when $k\geq\kappa_B\frac{\mu(\Omega)-\kappa-\kappa_B}{\exp(-\alpha n)}$. Set $\Sigma^\varphi=\left\{y_1,\ldots,y_k\right\}$ and $\Sigma_\varphi=\left\{y_1,\ldots,y_k\right\}$ for $\Xi_0=\Omega^\varphi\cap\Omega_B$ and $\Xi_0=\Omega_\varphi\cap\Omega_B$ appropriately.

For each $1\leq j< i \leq k$ we have $y_j\not\in y_i B^{(s,t)}$. Recall Lemma \ref{lem:disjoint-bowen} 
\begin{equation*}
{B_0^{(s,t)}}^{-1} B_0^{(s,t)}\subseteq B^{(s,t)}
\end{equation*}
This implies for $y_i,y_j$ as above that
$y_i B_0^{(s,t)}\cap y_j B_0^{(s,t)}=\emptyset$. We conclude that the sets $\Sigma^\varphi$ and $\Sigma_\varphi$ have all the properties (a), (b) and (c) from above.

\paragraph{The $m$-measure of the union of Bowen balls.}
Set
\begin{align*}
&\mathcal{D}^\varphi=\coprod_{y\in\Sigma^\varphi} y B_0^{(s,t)}\\
&\mathcal{D}_\varphi=\coprod_{y\in\Sigma_\varphi} y B_0^{(s,t)}
\end{align*}
Both unions are disjoint ones because of property (b) above. Because $B_0$ is an $(a,h)$-homogeneous neighborhood we have a lower bound on the $m$-measure
\begin{equation*}
m(\mathcal{D}^\varphi),m(\mathcal{D}_\varphi)
\geq \kappa_B\frac{\mu(\Omega)-\kappa-\kappa_B}{\exp(-\alpha n)} 
C_{B_0,h} \exp(-n h)
\end{equation*}
At this point we can already optimize $\kappa_B$ to make this measure as large as possible by setting $\kappa_B=\frac{\mu(\Omega)-\kappa}{2}$. We then have
\begin{equation}\label{eq:D-m-bound}
m(\mathcal{D}^\varphi),m(\mathcal{D}_\varphi)
\geq \frac{\left(\mu(\Omega)-\kappa\right)^2}{4} 
C_{B_0,h}\exp\left(-n (h-\alpha)\right)
\end{equation}

Because $\mathcal{D}^\varphi$ and $\mathcal{D}_\varphi$ are union of $B_0$-Bowen balls of points in $Y^\varphi$ and $Y_\varphi$ respectively, we have by Lemma \ref{lem:semicont-bound} that if $\varphi$ is semicontinuous in the appropriate way then 
\begin{align*}
&\forall y\in D^\varphi\colon A^{(s,t)}\varphi^{B_0}\,(y)\geq  \frac{\mu(\varphi)-(1-\kappa)}{\kappa}\\
&\forall y\in D_\varphi\colon A^{(s,t)}\varphi_{B_0}\,(y)\leq  \frac{\mu(\varphi)}{\kappa}
\end{align*}

We now apply Theorem \ref{thm:LD-specgap} to $\varphi^{B_0}$ and $\varphi_{B_0}$ respectively.
\begin{enumerate}[label=(\roman*)]
\item If $\varphi$ is lower-semicontinuous and
\begin{equation*}
\kappa \geq \frac{1-\mu(\varphi)}{1- m(\varphi^{B_0})};\quad
\theta\leq\frac{\left[{\mu(\varphi)-(1-\kappa)}\right]/{\kappa}- m(\varphi^{B_0})}{1-m(\varphi^{B_0})}
\end{equation*}
then
\begin{align*}
m\left(\mathcal{D}^\varphi\right)
&\leq \frac{\log\theta}{\log\lambda}\exp{D\left(\left[{\mu(\varphi)-(1-\kappa)}\right]/{\kappa} \|\, m(\varphi^{B_0})\right)}\\
&\times\exp\left[-n(-\log\lambda)\frac{D\left(\left[{\mu(\varphi)-(1-\kappa)}\right]/{\kappa} \|\, \theta+(1-\theta)m(\varphi^{B_0})\right)}{-\log\theta}\right]\\
\end{align*}
\item If $\varphi$ is upper-semicontinuous
and
\begin{equation*}
\kappa \geq \frac{\mu(\varphi)}{m(\varphi^{B_0})};\quad
\theta\leq\frac{m(\varphi^{B_0})-\mu(\varphi)/{\kappa}}{1-m(\varphi^{B_0})}
\end{equation*}
then
\begin{align*}
m(\mathcal{D}_\varphi)
&\leq \frac{\log\theta}{\log\lambda}\exp{D\left(\mu(\varphi)/\kappa \|\, m(\varphi_{B_0})\right)}\\
&\times\exp\left[-n(-\log\lambda)\frac{D\left(\mu(\varphi)/\kappa \|\, (1-\theta)m(\varphi_{B_0})\right)}{-\log\theta}\right]\\
&\leq\frac{\log\theta}{\log\lambda} \exp{D\left(\mu(\varphi)\|\, m(\varphi_{B_0})\right)}\\
&\times\exp\left[-n(-\log\lambda)\frac{D\left(\mu(\varphi)/\kappa \|\, (1-\theta)m(\varphi_{B_0})\right)}{-\log\theta}\right]
\end{align*}
\end{enumerate}
Inequalities (i) and (ii) combined with \eqref{eq:D-m-bound} imply the theorem.
\end{proof}

\section{Effective Rigidity of the Measure of Maximal Entropy}
In this section we discuss how an $L^2_0$-norm bound for the averaging operator implies that the push forward of a measure with very large entropy to a state space is close in weak-$*$ sense to the Haar measure. 

We deduce this result from a proposition about limits of sequences of measures with small Bowen balls on finer and finer scales. 

Note that our results \emph{do not} imply directly the uniqueness of the measure of maximal entropy. They just imply that all the measures of maximal entropy have a unique pushforward to the state space.

Although our methods are strong enough to show non-escape of mass for sequences of measure in some favorable situations
we do not discuss this in the current section. See \S\ref{sec:non-escape} for results about non-escape of mass.

\paragraph{Standing assumptions}
We recall our standing assumptions and expand them. Suppose $G$ is a second-countable locally compact topological group with a lattice $\Gamma<G$. Set $Y\coloneqq\faktor{\Gamma}{G}$. Let $m$ be the unique probability Haar measure on $Y$, fix $a\in G$ and consider the dynamical system defined by the right $a$ action on $Y$.

Let $K<G$ be a compact subgroup such that $N_G(K)<G$ is open and define the state space $X\coloneqq \faktor{Y}{K}$. Assume a norm gap for the averaging operator $\lambda\coloneqq \|T\|_{L^2_0(X)}<1$.

We assume \emph{in addition} that $Y\to X$ has property (M) and that the Haar measure is $a$-homogeneous, see Proposition \ref{prop:homogeneous}.

\begin{prop}[Limits of measures invariant under a fixed element]\label{prop:limit-tight}
Fix $a\in G$ and let $\{\mu_i\}_{i=1}^\infty$ be a sequence of $a$-invariant probability measures on $Y$ such that $\mu_i$ converges to a probability measure $\mu$ in the weak-$*$ topology. Let $h>0$ be a fixed real number.

Suppose that we have a family $\mathcal{F}$ of compact subset of $Y$ such that for any $\varepsilon>0$ there exists $\Omega\in\mathcal{F}$ such that
$\limsup_{i\to\infty} \mu_i(\Omega)\geq 1-\varepsilon$.
 
Assume further that for any compact subset $\Omega\in\mathcal{F}$ there is an identity neighborhood $B\subset G$ and a sequence of positive integers $t_i\to_{i\to\infty}\infty$ such that for all $i\in\mathbb{N}$
\begin{equation}\label{eq:omega-bowen-bound}
\int_\Omega \mu_i\left(yB^{(-t_i,t_i)}\right)\dif\mu_i(y)\leq C_\Omega e^{-2h t_i}
\end{equation}
for some $C_\Omega>0$.

Let $\varphi\colon Y\to[0,1]$ be a $\mathcal{X}$-measurable function, if $\varphi$ is continuous or it is a characteristic function of either a closed or an open subset of $X$ then the following holds.
\begin{enumerate}[label=(\alph*)]
\item If $1>\mu(\varphi)> m(\varphi)$ then for any
$\lambda^\mathbb{N}\ni\theta<\frac{\mu(\varphi)-m(\varphi)}{1-m(\varphi)}$
we have
\begin{equation*}
\frac{D\left(\mu(\varphi) \,\|\, \theta+(1-\theta)m(\varphi)\right)}{-\log\theta}
\leq\frac{h_m(a)-h}{-\log\lambda}
\end{equation*}
\item If $0<\mu(\varphi)< m(\varphi)$ then for any 
$\lambda^\mathbb{N}\ni\theta<\frac{m(\varphi)-\mu(\varphi)}{m(\varphi)}$
we have
\begin{equation*}
\frac{D\left(\mu(\varphi) \,\|\, (1-\theta)m(\varphi)\right)}{-\log\theta}
\leq\frac{h_m(a)-h}{-\log\lambda}
\end{equation*}
\end{enumerate}
\end{prop}
\begin{remark}
This result generalizes Linnik's approach to equidistribution using large deviation \cite{LinnikBook}. It should be compared with the method of \cite{PointsOnSphere} which served as a base for our work. The results of \cite{PointsOnSphere} essentially use a variation of the proposition above for finite state spaces, $h=h_m(a)$ and $\theta\to0$. 

A similar result for $h=h_m(a)$ and $Y\coloneqq\lfaktor{\mathbf{PGL}_2(\mathbb{R})}{\mathbf{PGL}_2(\mathbb{R})}$ is proven in \cite[Theorem 4.2]{ELMVPGL2}. This last result is different from ours in a few significant ways. Instead of using a norm gap it uses an effective description of the uniqueness of measure of maximal entropy; it is stated for an action by an element $a$ in the real place, which our results do not cover; and it applies for the whole space $Y$ and not just to a state space. Moreover, \cite[Theorem 4.2]{ELMVPGL2} allows to control escape of mass, which our methods can also accomplish.

The downside of \cite[Theorem 4.2]{ELMVPGL2} is that it applies only to $h=h_m(a)$. Our result is effective, it provides useful information in a limited range of $h<h_m(a)$. Moreover, one can give a rate of convergence in our result -- that is essentially Theorem \ref{thm:effective}.
\end{remark}

\begin{remark}
We can take $\mathcal{F}$ to be the collection of \emph{all} compact subsets of $Y$, this would require a bound on the average measure of Bowen balls in any compact subset. 

We show that for any $\varepsilon>0$ there exists a compact subset $\Omega\subset Y$ such that 
$\limsup_{i\to\infty} \mu_i(\Omega)\geq 1-\varepsilon$. This follows by a standard argument. Specifically, take a sequence of closed balls, $\{C_{r_k}\}_{k=1}^\infty$, around a fixed point in $Y$ with a strongly monotonic sequence of radii $\{r_k\}_{k=1}^\infty$ converging to infinity.  We then have that all the set $\partial C_{r_k}$ are mutually disjoint, hence $\sum_{k=1}^\infty \mu(\partial C_{r_k})=\mu(\bigcup_{k=1}^\infty \partial C_{r_k})<\infty$. This implies that $\mu(\partial C_{r_k})\to_{k\to\infty}=0$, on the other hand $\mu( C_{r_k})\to_{k\to\infty}=1$, hence $\mu(C_{r_k}^\circ)=\mu(C_{r_k})- \mu(\partial C_{r_k})\to_{k\to\infty}=1$. Taking $k$ large enough we get the required compact set $\Omega=C_{r_k}$. We actually have a stronger inequality
\begin{equation*}
\liminf_{i\to\infty} \mu_i(\Omega) \geq \liminf_{i\to\infty} \mu_i(\Omega^\circ) \geq \mu(\Omega^\circ)\geq1-\varepsilon
\end{equation*}

If we take $\mathcal{F}$ to be the collection of all compact sets then our result bears similarities to \cite[Proposition 3.3]{ELMVPeriodic} which under mildly stronger assumption shows an entropy lower bound $h_\mu(a)\geq h$. In spite of the similarities, it seems one cannot simply deduce the proposition above and \cite[Proposition 3.3]{ELMVPeriodic} from each other. In particular, in the $S$-algebraic setting our proposition gives strong information for $h$ close to $h_m(a)$ and no information at all for small values of $h$. On the other hand, \cite[Proposition 3.3]{ELMVPeriodic} gives non-trivial information in the whole range $0<h\leq h_m(a)$, yet this information is weaker then ours. If the measure $\mu$ is invariant under additional elements of $G$ commuting with $a$, then having positive entropy alone can imply very strong results regarding the measure using the measure rigidity results of \cite{QuantumLindenstrauss, EKL, EntropySArithmetic}.
\end{remark}
\begin{proof}
We prove the theorem for the case $\mu(\varphi)>m(\varphi)$, the proof of the second case is analogues. 

Assume first that $\varphi$ is continuous and compactly supported.

For any $\epsilon>0$ let $\Omega\in \mathcal{F}$ be a compact set such that $\liminf_{i\to\infty}\mu(\Omega)>1-\varepsilon$. Let $B$ and $\{t_i\}_{i=1}^\infty$ the identity neighborhood and the integer sequence corresponding to $\Omega$ in the hypothesis of the proposition.
Set in accordance with Theorem \ref{thm:effective}
\begin{equation}\label{eq:alpha-eta-ineq1}
\alpha\coloneqq-\frac{1}{2 t_i}\log\int_{\Omega} \mu_i\left(yB^{(-t_i,t_i)}\right)\dif\mu_i(y) \geq h -\frac{1}{2t_i}\log(C_\Omega) 
\end{equation} 

Fix $h_0>h_m(a)$.
We use the standing assumption that the Haar measure is $a$-homogeneous to generate a descending sequence of $(a,h_0)$-homogeneous identity neighborhood $B_0^j\subset N_G(K)$ satisfying the conditions of Theorem \ref{thm:effective} and such that $\bigcap_{j=1}^\infty B_0^j=\{e\}$.

As $\varphi$ is continuous and compactly supported it is actually uniformly continuous, hence $\varphi^{B_0^j}\to_{j\to\infty}\varphi$ pointwise. By the monotone convergence theorem $\lim_{j\to\infty} m(\varphi^{B_0^j})=m(\varphi)$. 

Recall that $\mu_i(\varphi)\to_{i\to\infty}\mu(\varphi)\in[0,1]$. We are dealing with the case $\mu(\varphi)>m(\varphi)$, hence for all $i$ and $j$ large enough one has $\mu_i(\varphi)>m(\varphi^{B_0^j})$.

By Theorem \ref{thm:effective} for any $\kappa$ and $\theta\in\lambda^\mathbb{N}$ satisfying the following
\begin{equation*}
\mu_i(\Omega)>\kappa \geq \frac{1-\mu_i(\varphi)}{1- m(\varphi^{B_0^j})};\quad
\theta\leq\frac{\left[{\mu_i(\varphi)-(1-\kappa)}\right]/{\kappa}- m(\varphi^{B_0^j})}{1-m(\varphi^{B_0^j})}
\end{equation*}
and for all $i$ and $j$ large enough we have
\begin{align*}
&\frac{D\left(\left[{\mu_i(\varphi)-(1-\kappa)}\right]/{\kappa} \,\|\, \theta+(1-\theta)m(\varphi^{B_0^j})\right)}{-\log\theta}
\leq\frac{h_0-h}{-\log\lambda}+\frac{1}{2t_i}\log(C_\Omega) \\
&+\frac{1}{2t_i}\Big[\log\left(\frac{\log\theta}{\log\lambda}\right) D\left(\mu_i(\varphi)\,\|\,m(\varphi^{B_0^j})\right)\\
&-2\log\frac{\mu_i(\Omega)-\kappa}{2}
-\log C_{B_0^j,h_0})\Big]
\end{align*}
where we have also applied inequality \eqref{eq:alpha-eta-ineq1}.

Recall that
$\limsup_{i\to\infty} \mu_i(\Omega) \geq 1-\varepsilon$.
Now let $i\to\infty$ along a subsequence attaining the $\limsup$ in the inequality for $\Omega$, then for any $j$ large enough and any $\kappa$ and $\theta\in\lambda^\mathbb{N}$ satisfying
\begin{equation*}
1-\varepsilon>\kappa > \frac{1-\mu(\varphi)}{1- m(\varphi^{B_0^j})};\quad
\theta <\frac{\left[{\mu(\varphi)-(1-\kappa)}\right]/{\kappa}- m(\varphi^{B_0^j})}{1-m(\varphi^{B_0^j})}
\end{equation*}
we have
\begin{equation*}
\frac{D\left(\left[{\mu(\varphi)-(1-\kappa)}\right]/{\kappa} \,\|\, \theta+(1-\theta)m(\varphi^{B_0^j})\right)}{-\log\theta}
\leq\frac{h_0-h}{-\log\lambda}
\end{equation*}

The claim for $\varphi$ continuous and compactly supported follows by first taking the limits $j\to\infty$, $\varepsilon\to0$, $\kappa\to 1$,  and at last $h_0\to h_m(a)$. 

For $\varphi$ continuous not necessarily compactly supported one can take a monotonic sequence of compactly supported functions $0\leq\varphi_k\leq\varphi$ which converge pointwise to $\varphi$. Then the claim for $\varphi$ follows from the claim for the functions $\varphi_k$ by taking the limit $k\to\infty$ and using the monotone convergence theorem. 

If $\varphi$ is the characteristic of an open or a closed subset of $X$ then similarly it can be approximated by a sequence of continuous functions $\varphi_k\colon X\to[0,1]$ converging pointwise to $\varphi$. By Lebesgue's dominated convergence theorem for every measure $\nu$ on $Y$ one has $\lim_{k\to\infty} \int_Y \varphi_k \dif\nu\to_{k\to\infty} \int_Y \varphi \dif\nu$. The result now follows from the claim for continuous functions by taking the limit  $k\to\infty$.
\end{proof}

To pass from average mass of Bowen balls on a prescribed scale to a result regarding the classical notion of entropy we require the following standard lemma.
\begin{lem}\label{lem:brin-katok-ergodic}
Let $\mu$ be an $a$-invariant \emph{ergodic} measure on $Y$.
Then for every $\varepsilon>0$ there exists a compact subset $\Omega\subseteq Y$, with $\mu(\Omega)>1-\varepsilon$, and an identity neighborhood $B\subseteq G$ such that for all $t>N_\Omega\in\mathbb{N}$ and all $y\in\Omega$
\begin{equation*}
\mu(yB^{(-t,t)}) \leq  e^{-2t(h_\mu(a)-\varepsilon)}
\end{equation*}
\end{lem}
\begin{proof}
This follows from an extension of the Brin-Katok theorem \cite{BrinKatok} to a second countable metric space, Theorem \ref{thm:bk-second-countable}.

Fix a left invariant metric $d$ on $G$ and let $\widetilde{d}$ be the induced metric on $Y$. Denote by $\widetilde{B}_r(y)$ be the open $\widetilde{d}$-ball of radius $r>0$ centered at $y\in Y$. Let $B_r\subset G$ be the ball of radius $r>0$ around the identity. We claim that for all $y=\Gamma g\in Y$ and $r>0$ we have $\widetilde{B}_r(y)=y B_r$. To see that we unwind the definition of $\widetilde{d}$
\begin{align*}
\widetilde{d}(\Gamma g', \Gamma g) < r &\Leftrightarrow \left(\exists \gamma\in\Gamma\right)\, d(\gamma g',  g) <r 
\Leftrightarrow  \left(\exists \gamma\in\Gamma\right)\, d(g^{-1}\gamma g',  e) <r\\
&\Leftrightarrow  \left(\exists \gamma\in\Gamma\right)\,{g}^{-1}\gamma g\in B_r
\Leftrightarrow  \left(\exists \gamma\in\Gamma\right) \gamma g'\in g B_r\\
&\Leftrightarrow \Gamma g'\in \Gamma g B_r
\end{align*}

For any $s\leq 0 \leq t$  define 
\begin{equation*}
\widetilde{B}^{(s,t)}_r(y)\coloneqq\left\{y'\in Y \mid \forall s \leq n \leq t\colon \widetilde{d}(y'a^{-n},y a^{-n})<r\right\}
\end{equation*}

By Theorem \ref{thm:bk-second-countable} for almost all $y\in Y$
\begin{equation*}
\lim_{r\to 0} \liminf_{t\to\infty} \frac{-\log \mu(\widetilde{B}^{(-t,t)}_r(y))}{2t}\geq h_\mu(a)
\end{equation*}
This implies by a standard argument that for every $\varepsilon>0$ there exists a compact subset $\Omega\subseteq Y$ with $\mu(\Omega)>1-\varepsilon$, a radius $r_\varepsilon>0$ and an integer $N_{r_\varepsilon}\in\mathbb{N}$ such that for all $t\geq N_{r_\varepsilon}$, for all $r<r_\varepsilon$ and all $y\in \Omega$
\begin{equation*}\label{eq:single-bowen-ineq}
\frac{-\log \mu(\widetilde{B}^{(-t,t)}_r(y))}{2t}\geq h_\mu(a)-\varepsilon
\end{equation*} 

Fix $B=B_{r_\varepsilon/2}$. 
For any $y\in Y$ we have $y B^{(-t,t)}\subseteq \widetilde{B}_{r_\varepsilon/2}^{(-t,t)}(y)$. Set $N_\Omega\coloneqq N_{r_\varepsilon/2}$.
\end{proof}
\begin{cor}\label{cor:entropy-ergodic}
Let $\mu$ be an $a$-invariant \emph{ergodic} probability measure then. Let $\varphi\colon Y\to[0,1]$ be a $\mathcal{X}$-measurable function. If $\varphi$ is continuous or it is the characteristic function of either a closed or an open subset of $X$ then the following holds.
\begin{enumerate}[label=(\alph*)]
\item If $1>\mu(\varphi)> m(\varphi)$ then for any $\lambda^\mathbb{N}\ni\theta<\frac{\mu(\varphi)-m(\varphi)}{1-m(\varphi)}$
we have
\begin{equation*}
\frac{D\left(\mu(\varphi) \,\|\, \theta+(1-\theta)m(\varphi)\right)}{-\log\theta}
\leq\frac{h_m(a)-h_\mu(a)}{-\log\lambda}
\end{equation*}
\item If $0<\mu(\varphi)< m(\varphi)$ then for any 
$\lambda^\mathbb{N}\ni\theta<\frac{m(\varphi)-\mu(\varphi)}{m(\varphi)}$
we have
\begin{equation*}
\frac{D\left(\mu(\varphi) \,\|\, (1-\theta)m(\varphi)\right)}{-\log\theta}
\leq\frac{h_m(a)-h_\mu(a)}{-\log\lambda}
\end{equation*}
\end{enumerate}
\end{cor}
\begin{proof}
This follows directly from Proposition \ref{prop:limit-tight} with the constant sequence $\mu_i=\mu$ for all $i\in\mathbb{N}$. The family of compact sets we take in the hypothesis of Proposition \ref{prop:limit-tight} is the collection of all compact sets for all $\varepsilon>0$ provided by Lemma \ref{lem:brin-katok-ergodic}.
\end{proof}

Much stronger results then the following are known in most cases, \cite[\S 9]{MargulisTomanov}, \cite[Corollary 7.10]{Pisa}. Nevertheless, we think it is beneficial to state and prove this result in the general setting we are working in.
\begin{cor}
Under our standing assumptions every $a$-invariant and ergodic measure $\mu$ on $Y$ for which $h_\nu(a)\geq h_m(a)$ has the same pushforward to the state space $X$.
\end{cor}
\begin{proof}
Let $\varphi\colon Y\to[0,1]$ be any $\mathcal{X}$-measurable continuous function with $0<\mu(\varphi)<1$. 
By Corollary \ref{cor:entropy-ergodic} if $\mu(\varphi)>\theta+(1-\theta)m(\varphi)$ for some $\theta\in\lambda^\mathbb{N}$ then 
\begin{equation*}
D\left(\mu(\varphi) \,\|\, \theta+(1-\theta)m(\varphi)\right)\leq0
\end{equation*}
which implies $\mu(\varphi)=\theta+(1-\theta)m(\varphi)$. Similarly, we have $\mu(\varphi)=(1-\theta)m(\varphi)$ if $\mu(\varphi)<(1-\theta)m(\varphi)$ for some $\theta\in\lambda^\mathbb{N}$.

We have shown that we must have $(1-\theta)m(\varphi)\leq\mu(\varphi)\leq\theta+(1-\theta)m(\varphi)$ for all $\theta\in\lambda^\mathbb{N}$. Taking $\theta\to 0$ we have that $\mu(\varphi)=m(\varphi)$ which proves the statement.
\end{proof}

\begin{thm}\label{thm:rigidity-max-entr}
Assume in addition to our standing assumptions that the Haar measure is a measure of maximal entropy.
Let $\mu$ be an $a$-invariant probability measure then. Let $\varphi\colon Y\to[0,1]$ be a continuous $\mathcal{X}$-measurable function, then for any $\theta\in\lambda^\mathbb{N}$
\begin{equation}\label{eq:pinsker}
|\mu(\varphi)-m(\varphi)|\leq\sqrt{\frac{h_m(a)-h_\mu(a)}{-2\log\lambda}}\sqrt{-\log{\theta}}+\theta
\end{equation}
\end{thm}
\begin{remark}
In  the $S$-arithmetic setting, 
we know that the Haar measure is a measure of maximal entropy, \cite[\S 7]{Pisa} and \cite[\S 9]{MargulisTomanov}. If the horospherical subgroups corresponding to $a$ generate the whole group then it actually a unique measure of maximal entropy.
\end{remark}

\begin{proof}
To pass from the Kullback-Leibler divergence to the usual absolute value we use  Pinsker's inequality, which in our case reduces to
\begin{equation*}
|a-b|\leq\sqrt{\frac{1}{2} D(a\|\,b)}
\end{equation*}
for all $0\leq a,b\leq 1$.

Assume first that $\mu$ is ergodic and that $0<\mu(\varphi)<1$.
If $\mu(\varphi)>\theta+(1-\theta)m(\varphi)$ we have from Corollary \ref{cor:entropy-ergodic} and Pinsker's inequality
\begin{align*}
\left|\mu(\varphi)-m(\varphi)-\theta(1-m(\varphi))\right|
&\leq \sqrt{\frac{1}{2} D\left(\mu(\varphi) \,\|\, \theta+(1-\theta)m(\varphi)\right)}\\
&\leq \sqrt{\frac{h_m(a)-h_\mu(a)}{-2\log\lambda}}\sqrt{-\log{\theta}}
\end{align*}
hence
\begin{align}
\left|\mu(\varphi)-m(\varphi)\right| \label{eq:ergodic-pinsker}
&\leq \sqrt{\frac{h_m(a)-h_\mu(a)}{-2\log\lambda}}\sqrt{-\log{\theta}}+|\theta(1-m(\varphi))|\\ \nonumber
&\leq \sqrt{\frac{h_m(a)-h_\mu(a)}{-2\log\lambda}}\sqrt{-\log{\theta}}+\theta \nonumber
\end{align}

We can deduce in a similar manner inequality \eqref{eq:ergodic-pinsker} also in the case that
$\mu(\varphi)<(1-\theta)m(\varphi)$.

The only case we have left to deal in the ergodic case is when $(1-\theta)m(\varphi)<\mu(\varphi)<\theta+(1-\theta)m(\varphi)$, but then one has the stronger inequality
\begin{equation*}
-\theta\leq-\theta m(\varphi)\leq\mu(\varphi)-m(\varphi)\leq \theta(1-m(\varphi))\leq\theta
\end{equation*}
We have conclude the proof when $\mu$ is ergodic, $0<\mu(\varphi)<1$  and $\theta\in\lambda^\mathbb{N}$ arbitrary.
To remove the restriction on $\mu(\varphi)$ look for every $1>\varepsilon>0$ at $\varphi_\varepsilon=\varepsilon+(1-2\varepsilon)\varphi$. We have $\varepsilon\leq\mu(\varphi_\varepsilon)\leq1-\varepsilon$. Apply inequality \eqref{eq:pinsker} to $\varphi_\varepsilon$ and take $\varepsilon\to 0$ to see that the same holds to $\varphi$.

To pass to a general $\mu$ we use the ergodic decomposition. Let $\mathcal{E}\subseteq\mathcal{Y}$ the $\sigma$-algebra of $a$-invariant sets. The measure disintegration of $\mu$ with respect to $\mathcal{E}$ is an ergodic decomposition, i.e.\ $\mu_y^\mathcal{E}$ is $a$-invariant and ergodic for $\mu$ almost every $y\in Y$.

Now fix $\theta\in\lambda^\mathbb{N}$ and integrate \eqref{eq:ergodic-pinsker} over the ergodic components
\begin{align*}
\left|\mu(\varphi)-m(\varphi)\right|&=\left|\int_Y \left[\mu_y^\mathcal{E}(\varphi)-m(\varphi)\right]\dif\mu(y)\right|
\leq\int_Y \left|\mu_y^\mathcal{E}(\varphi)-m(\varphi)\right|\dif\mu(y)\\
&\leq \int_Y {\sqrt{\frac{h_m(a)-h_{\mu_y^\mathcal{E}}(a)}{-2\log\lambda}}}\dif\mu(y) \sqrt{-\log{\theta}}+\theta\\
&\leq  \sqrt{\int_Y \frac{h_m(a)-h_{\mu_y^\mathcal{E}}(a)}{-2\log\lambda}\dif\mu(y)} \sqrt{-\log{\theta}}+\theta\\
&=  \sqrt{\frac{h_m(a)-h_{\mu}(a)}{-2\log\lambda}} \sqrt{-\log{\theta}}+\theta\\
\end{align*}
Where we have used the concavity of the function $x\mapsto \sqrt{x}$.
\end{proof}

\begin{cor}\label{cor:max-entrp-rig-simple}
Under the hypothesis of the Theorem \ref{thm:rigidity-max-entr} there is for every $\varepsilon>0$ an explicitly computable constant $C_{\varepsilon} >0$ such that
\begin{equation}\label{eq:sqrt-bound}
\left|h_m(a)-h_\mu(a)\right|\leq C_{\varepsilon} \left(h_m(a)-h_\mu(a)\right)^{1/2-\varepsilon}\left[(-\log\lambda)^{-1/2}+2\right]
\end{equation}
Where $C_{\varepsilon}$ depends \emph{only} on $\varepsilon$ and $h_m(a)$.
\end{cor}
\begin{proof}
Denote $D\coloneqq h_m(a)-h_\mu(a)$.
If $\sqrt{D}>\lambda$ take $\theta=\lambda$ in Theorem \ref{thm:rigidity-max-entr}
\begin{align*}
|\mu(\varphi)-m(\varphi)|
&\leq\sqrt{\frac{D}{-2\log\lambda}}\sqrt{-\log{\theta}}+\theta
=\sqrt{\frac{D}{2}}+\lambda
<\sqrt{D}\left(1/\sqrt{2}+1\right)\\
&<2\sqrt{D}
\end{align*}

Otherwise, let $k\in\mathbb{N}$ such that $\lambda^{k+1}\leq\sqrt{D}\leq \lambda^k$ and take $\theta=\lambda^{k+1}$. We have 
\begin{align*}
\theta&\leq  \sqrt{D}\\
-\log\theta&\leq-\frac{1}{2}\log D-\log\lambda
\end{align*}
and
\begin{align}
|\mu(\varphi)-m(\varphi)|
&\leq\sqrt{\frac{D\log D}{4\log\lambda} +\frac{D}{2}}+\sqrt{D}
\label{eq:sqrt-log-bound}\\
&<\sqrt{\frac{D\log D}{4\log\lambda}} +\sqrt{D}/\sqrt{2}+\sqrt{D}
\nonumber
\end{align}

For all $\varepsilon>0$ we have $\frac{\log r}{r^{-\varepsilon}}\to_{r\to 0^+}0$, hence there is $C_\varepsilon'>0$ such that $-\log r\leq C_\varepsilon' r^{-\varepsilon}$ for all $0\leq r\leq \sqrt{h_m(a)}$.
\begin{equation*}
|\mu(\varphi)-m(\varphi)|
\leq \sqrt{C_\varepsilon'}/2\,D^{1/2-\varepsilon}
\left[\left(-\log\lambda\right)^{-1/2}+2\right]
\end{equation*}

The claim follows by taking $C_\varepsilon=\max\{\sqrt{C_\varepsilon'}/2,1\}$ and applying inequality \eqref{eq:sqrt-bound}  and \eqref{eq:sqrt-log-bound}.
\end{proof}

The following should be compared with \cite[Theorem 1.1]{Ruhr}, see the discussion in \S\ref{sec:intro-max}.
\begin{thm}\label{thm:rigidity-max-alg}
Let $\mathbf{G}$ be a connected simply connected absolutely almost simple linear algebraic group defined over $\mathbb{Q}$. Let $S$ be a finite set of places for $\mathbb{Q}$ including $\infty$, such that $G_S\coloneqq \prod_{v\in S} \mathbf{G}(\mathbb{Q}_v)$ is not compact. 

Let $p\in S$ be a finite place such that $G_p\coloneqq\mathbf{G}(\mathbb{Q}_p)$ is not compact. Fix $a\in G_p\cap\mathbf{G}(\mathbb{Q})$ $\mathbb{Q}_p$-regular and semisimple and not contained in a compact subgroup. 

Denote by $A<G_p$ the maximal split subtorus of a maximal torus including the element $a$. Let $\Delta$ be the affine Bruhat-Tits building associated to $G_p$.
Denote by $H<G_p$ the pointwise stabilizer of the apartment corresponding to $A$ in $\Delta$. 

Let $\Gamma<G_S$ be a congruence lattice and denote by $m$ the Haar probability measure on $\lfaktor{\Gamma}{G_S}$.

Let $\varphi\colon\sdfaktor{\Gamma}{G_S}{H}\to[0,1]$ be a continuous function that is smooth in the $p$-adic coordinate, i.e.\ when considered as a function on $\lfaktor{\Gamma}{G_S}$ it is invariant under some compact open subgroup $K_0<G_p$ containing $H$. Then for any $\varepsilon>0$ there is an explicitly computable constant $C>0$ depending only on $\varepsilon$, $G_p$, $K_0$ and $a$ such that for any $a$-invariant measure $\mu$ we have
\begin{equation*}
|\mu(\varphi)-m(\varphi)|\leq C \left(h_m(a)-h_\mu(a)\right)^{1/2-\varepsilon}
\end{equation*}
\end{thm}
\begin{remark}
In this setting the Haar measure is a measure of maximal entropy, \cite[\S 9]{MargulisTomanov} and \cite[\S 7.8]{Pisa}
\end{remark}
\begin{proof}
For the $S$-arithmetic quotients we consider there is an $L^2$-norm gap for the averaging operator $T$ for \emph{any} state space of the form $\faktor{Y}{V}$ where $V<G_p$ is a compact-open subgroup. The specific result we use appears in \cite[\S 1.6]{GMO}. See the discussion in \S\ref{sec:intro-max} regarding bounds on matrix coefficients.

If $\mu$ is $a$ invariant then it is $a^n$ invariant for all $n\in\mathbb{N}$. Let $\Delta$ be the affine Bruhat-Tits building associated to $G_p$. Let $\mathcal{A}\subset\Delta$ be the apartment corresponding to $A$.
For any special vertex $v\in\mathcal{A}$ we have associated in \S \ref{sec:walks-buildings} an arrow subgroup  $K_{(v,a^{-2n}.v)}$ for the $a^{-2n}$ action which stabilizes the rhomboid bounded by the corners $v$ and $a^{-2n}.v$. 

Fix a special vertex $v_0\in\mathcal{A}$ and consider the arrow subgroups $K_n\coloneqq K_{(a^n.v_0,a^{-n}.v_0)}$. The corresponding rhomboids $\mathcal{R}_n$ are a monotone ascending sequence of convex subsets of $\mathcal{A}$ which cover the apartment as $n\to\infty$. Hence $\bigcap_{n=1}^\infty K_n=H$ -- the pointwise stabilizer of the apartment. By Proposition \ref{prop:M-algebraic} all the state spaces $\faktor{Y}{K_n}$ have property (M) with respect to the right action by $a^{2n}$. 

If $\varphi$ is invariant under $K_0$ then the sequence $\left\{K_n\setminus K_0\right\}_{n=1}^\infty$ is a descending sequence of compact sets having empty intersection, hence there exists $n\in\mathbb{N}$ such that $K_n\setminus K_0=\emptyset \Leftrightarrow K_n\subseteq K_0$ and $\varphi$ is $K_n$ invariant. Let $T_{a^{2n}}$ be the averaging operator associated to $\faktor{Y}{K_n}$ and the right action by $a^{2n}$. Set $\lambda\coloneqq \|T_{a^{2n}}\|_{L^2_0(X)}$.

We have $\lambda<1$ by \cite[\S 1.6]{GMO}. By Corollary \ref{cor:max-entrp-rig-simple} 
\begin{align*}
|\mu(\varphi)-m(\varphi)|
&\leq C_\varepsilon (h_m(a^{2n})-h_\mu(a^{2n}))^{1/2-\varepsilon}\left[(\log\lambda)^{-1/2}+2\right]\\
&=C_\varepsilon\sqrt{2n}\left[(\log\lambda)^{-1/2}+2\right] (h_m(a)-h_\mu(a))^{1/2-\varepsilon}\\
\end{align*}
Set now $C=C_\varepsilon\sqrt{2n}\left[(\log\lambda)^{-1/2}+2\right]$.
\end{proof}

\section{Non-Escape of Mass}\label{sec:non-escape}
There are two flavors of related results, one where we have a sequence of measures with average mass of Bowen balls exponentially small on finer and finer scale and the second where we have a sequence of measures with high metric entropy. We shall focus only on the second case, although results of the first type can be deduced using our methods.

See \S \ref{sec:intro-non-escape} for a discussion about the relation about our theorem for non-escape of mass and previously known results in the archimedean place.

\begin{prop}\label{prop:bigset-entropy}
Assume in addition to our standing assumptions that the Haar measure is a measure of maximal entropy.
Let $\mu$ be an $a$-invariant probability measure. For any closed subset $F\subset X$ we have for all $\theta\in\lambda^\mathbb{N}$
\begin{align*}
1-\mu(F)
\le &\frac{\log\theta}{\log\left(\theta+(1-\theta)(1-m(F))\right)}
\left[\frac{h_m(a)-h_\mu(a)}{-\log\lambda}+\frac{1}{-\log\theta} \right]\\
&+ \theta+(1-\theta)(1-m(F))
\end{align*}
\end{prop}

\begin{proof}
Let $\mathcal{E}\subseteq\mathcal{Y}$ the $\sigma$-algebra of $a$-invariant sets. The conditional measure $\mu_y^\mathcal{E}$ is well-defined, $a$-invariant and ergodic for all $y\in Y_0$, with $\mu(Y_0)=1$. Fix $\theta\in\lambda^\mathbb{N}$. Set $Y_1\coloneqq\left\{y\in Y_0 \mid \mu_y^\mathcal{E}(F)<(1-\theta)m(F) \right\}$.
By Corollary \ref{cor:entropy-ergodic} for any $y\in Y_1$
\begin{equation}\label{eq:ergodic-C-KL}
\frac{D\left(\mu_y^\mathcal{E}(F) \,\|\, (1-\theta)m(F)\right)}{-\log\theta}
\leq\frac{h_m(a)-h_{\mu_y^\mathcal{E}}(a)}{-\log\lambda}
\end{equation}

We now expand $D\left(\mu_y^\mathcal{E}(F) \,\|\, (1-\theta)m(F)\right)$ using the explicit definition of the Kullback-Leibler divergence. For any $a,b\in[0,1]$
\begin{align*}
D\left(a \,\|\, b\right)&=-a\log(b)-(1-a)\log(1-b)-H(a)\\
&\geq -a\log(b)-(1-a)\log(1-b)-1
\end{align*}
where $H(a)=-a\log(a)-(1-a)\log(1-a)\leq 1$ is the binary entropy function in natural base. Using this inequality and \eqref{eq:ergodic-C-KL} we have for all $y\in Y_1$
\begin{align*}
&(1-\mu_y^\mathcal{E}(F))\frac{\log\left(\theta+(1-\theta)(1-m(F))\right)}{\log\theta}\\
&\leq\mu_y^\mathcal{E}(F) \frac{\log\left((1-\theta)m(F)\right)}{\log\theta}+(1-\mu_y^\mathcal{E}(F))\frac{\log\left(\theta+(1-\theta)(1-m(F))\right)}{\log\theta}\\
&\leq\frac{h_m(a)-h_{\mu_y^\mathcal{E}}(a)}{-\log\lambda}+\frac{1}{-\log\theta} 
\end{align*}

Denote $c\coloneqq\frac{\log\left(\theta+(1-\theta)(1-m(F))\right)}{\log\theta}>0$.
Integrating over $Y_1$ we have
\begin{equation}\label{eq:Y1-bound}
\int_{Y_1}(1-\mu_y^\mathcal{E}(F))\dif\mu(y)
\leq c^{-1}\int_{Y_1}\frac{h_m(a)-h_{\mu_y^\mathcal{E}}(a)}{-\log\lambda}\dif\mu(y)+c^{-1}\mu(Y_1)\frac{1}{-\log\theta} 
\end{equation}

Set $Y_2\coloneqq\left\{ y\in Y_0 \mid \mu_y^\mathcal{E}(F)\geq(1-\theta)m(F) \right\}$, then
\begin{align}
\int_{Y_2}(1-\mu_y^\mathcal{E}(F))\dif\mu(y)
&\leq \mu(Y_2)\left[\theta+(1-\theta)(1-m(F))\right]
\nonumber\\
&\leq c^{-1}\int_{Y_2}\frac{h_m(a)-h_{\mu_y^\mathcal{E}}(a)}{-\log\lambda}+\mu(Y_2)\left[\theta+(1-\theta)(1-m(F))\right]
\label{eq:Y2-bound}
\end{align}

Summing the inequalities \eqref{eq:Y1-bound} and \eqref{eq:Y2-bound} and using the linearity of the entropy function on the space of measures we have
\begin{align*}
1-\mu(F)
&\leq c^{-1}\left[\frac{h_m(a)-h_\mu(a)}{-\log\lambda}+\mu(Y_1)\frac{1}{-\log\theta}\right]\\
&+\mu(Y_2)\left[\theta+(1-\theta)(1-m(F))\right]\\
&\leq  c^{-1}\left[\frac{h_m(a)-h_\mu(a)}{-\log\lambda}+\frac{1}{-\log\theta}\right]+\theta+(1-\theta)(1-m(F))
\end{align*}
\end{proof}

\begin{cor}[Non-escape of mass]\label{cor:non-escape}
Suppose we are given a sequence of $a$-invariant probability measures $\mu_i$ on $Y$. Set $h=\liminf_{i\to\infty} h_{\mu_i}(a)$.

If $\mu_i$ converges to a measure $\mu$ in the weak-$*$ topology then
\begin{equation*}
\mu(Y)\geq 1- \frac{h_m(a)-h}{-\log(\lambda)}
\end{equation*}
\end{cor}
\begin{proof}
Any compact subset $\Omega\subset Y$ is a closed subset of the one-point compactification of $Y$, hence $\mu(\Omega)\geq\limsup_{i\to\infty} \mu_i(\Omega)$. Fix a compact subset $\Omega\subset X$. Applying Proposition \ref{prop:bigset-entropy} to $\mu_i(\Omega)$ for all $i\in\mathbb{N}$ and taking $i\to\infty$ we have for all $\theta\in\lambda^\mathbb{N}$
\begin{align}
1-\mu(\Omega)
\leq&\frac{\log\theta}{\log\left(\theta+(1-\theta)(1-m(\Omega))\right)}
\left[\frac{h_m(a)-h}{-\log\lambda}+\frac{1}{-\log\theta} \right]
\label{eq:closed-big-limit}\\
&+ \theta+(1-\theta)(1-m(\Omega)) \nonumber
\end{align}

The space $Y$ is $\sigma$-compact, using the quotient map one can push forward  a countable collection of compact sets covering $Y$ to such a collection covering $X$, hence $X$ is also $\sigma$-compact.

If we now take an ascending sequence of compact sets $\Omega_k\subset X$ such $\bigcup_{k=1}^\infty \Omega_k=X$ then taking the limit $k\to\infty$ in inequality \eqref{eq:closed-big-limit} we have
\footnote{$\mu(X)=\mu(Y)$ and $m(Y)=1$}
\begin{align*}
1-\mu(Y)
\leq \frac{h_m(a)-h}{-\log\lambda}+\frac{1}{-\log\theta}+ \theta
\end{align*}
The claim follows by taking the limit $\theta\to0$.
\end{proof}

The following definition appears here in the form of \cite[\S2.2]{GMO}.
\begin{defi} \label{def:eta-function}
{\cite{Oh}.}
Let $\mathbf{G}$ be a connected absolutely almost simple group defined over $\mathbb{Q}$. Fix a finite rational prime $p$ and set $G_p\coloneqq \mathbf{G}(\mathbb{Q}_p)$. Let $A<G_p$ be a maximal $\mathbb{Q}_p$-split torus, $A^+$ a closed positive Weyl chamber.

Let $\Phi^+$ be a system of positive roots in the set of all non-multipliable roots of $G_p$ relative to $A^+$. Choose a maximal strongly orthogonal system $\mathcal{S}$ in $\Phi^+$ as defined and constructed in \cite{Oh}. Define
\begin{equation*}
\eta=\prod_{\alpha\in\mathcal{S}} \alpha
\end{equation*}
\end{defi}

\begin{remark}
We refer to \cite[Theorem A]{Oh} to a list of formulae  for $\eta$ for simple linear groups. We cite the relevant expression for $\mathbf{G}=\mathbf{SL}_n$ from \cite[Example 5.1]{COU}. Let $A<\mathbf{SL}_n(\mathbb{Q}_p)$ be the diagonal subgroup and set
\begin{equation*}
A^+=\left\{\diag(a_1,\ldots,a_n) \mid a_i\in\mathbb{N}, a_i>0 \textrm{ and }\forall 1\leq i\leq n-1\colon a_{i+1}|a_i  \right\}
\end{equation*}
then for all $a\in A^+$
\begin{equation*}
\eta(a)=\prod_{i=1}^{\lfloor n/2 \rfloor } \frac{|a_i|_p}{|a_{n+1-i}|_p}
\end{equation*}

For concreteness take $a=\diag(p^{n-1},p^{n-3},\ldots,p^{-(n-1)})$ then
\begin{equation*}
\begin{cases}
p^{2k^2} & \textrm{ if } n=2k\\
p^{2k(k+1)} & \textrm{ if } n=2k+1
\end{cases}
\end{equation*}
In contrast $p^{h_m(a)}=p^{2n(n+1)(n+2)/3}$. 
\end{remark}

\begin{thm}\label{thm:non-escape}
Let $\mathbf{G}$ be a connected simply connected absolutely almost simple linear algebraic group defined over $\mathbb{Q}$. Let $S$ be a finite set of places for $\mathbb{Q}$ including $\infty$, such that $G_S\coloneqq \prod_{v\in S} \mathbf{G}(\mathbb{Q}_v)$ is not compact. 

Let $p\in S$ be a finite place such that $G_p\coloneqq\mathbf{G}(\mathbb{Q}_p)$ is not compact, set $r=2$ if $\rank_{\mathbb{Q}_p} \mathbf{G}=1$ and $r=1$ if the rank is higher. 

Fix $a\in G_p\cap\mathbf{G}(\mathbb{Q})$ a $\mathbb{Q}_p$-regular and semisimple element not belonging to a compact subgroup. Let $A<G_p$ be a maximal split subtorus of a maximal torus containing $a$, set $A^+$ to be the closed Weyl chamber corresponding to $a$. Define $\eta$ with respect to these choices.  

Let $\Gamma<G_S$ be a congruence lattice and denote by $m$ the Haar probability measure on $\lfaktor{\Gamma}{G_S}$.

Suppose we are given a sequence of $a$-invariant probability measures $\mu_i$ on $Y$. Set $h=\liminf_{i\to\infty} h_{\mu_i}(a)$.
If $\mu_i$ converges to a measure $\mu$ in the weak-$*$ topology then
\begin{equation*}
\mu(Y)\geq 1-2r\frac{h_m(a)-h}{-\log\eta(a)}
\end{equation*}
\end{thm}

\begin{remark}
The factor of $r=2$ in the rank $1$ case is conjectured to be redundant \cite[Conjecture 2.15]{GMO}, this is related to the Ramanujan conjecture.
\end{remark}

\begin{proof}
Fix an arrow subgroup $K_\to<G_p$ corresponding to the right action by $a$ on $Y\coloneqq \lfaktor{\Gamma}{G_S}$. By Proposition \ref{prop:M-algebraic} the state space $X\coloneqq\faktor{Y}{K_\to}$ has property (M) for the right action by $a$, hence for the action by $a^n$ for all $n\in\mathbb{N}$. For $n\in\mathbb{N}$ let $T_{a^n}$ be the averaging operator induced by the $a^n$ action.

By \cite[Theorem 1.19 and Lemma 2.3]{GMO} for each $\varepsilon>0$ there exists $c=c(\Gamma,K_\to,\varepsilon)>0$ such that for all $n\in\mathbb{N}$
\begin{equation*}
\|T_{a^n}\|_{L^2_0(X)}\leq c \eta(a^n)^{\left(1/2-\varepsilon\right)/r}=c \eta(a)^{n\left(1/2-\varepsilon\right)/r}
\end{equation*}

Apply Corollary \ref{cor:non-escape} to the action by $a^n$ and use the fact that $h_{\mu_i}(a^n)=nh_{\mu_i}(a)$ to see that
\begin{equation*}
\mu(Y)\geq 1-\frac{nh_m(a)-nh}{-\left[n\left(1/2-\varepsilon\right)/r\right]\log\eta(a)-\log c}
\end{equation*}
The theorem follows by taking first the limit $n\to\infty$ and then $\varepsilon\to 0$.
\end{proof}

\appendix

\section{Local Entropy}
For lack of a good reference we present a proof of the following partial extension of the Brin-Katok \cite{BrinKatok} theorem to second countable metric spaces.
\begin{thm}\label{thm:bk-second-countable}
Let $(Y,d)$ be a second countable  metric space with a Borel $\sigma$-algebra $\mathcal{Y}$. Let $\mu$ be a Borel probability measure on $Y$ and $S\colon Y\to Y$ be an invertible continuous ergodic measure preserving transformation.
Define for every $y\in Y$, $r>0$ and $n\in\mathbb{N}$
\begin{equation*}
B_r^{(-n,n)}(y)=\left\{z\in Y \mid \forall -n\leq i\leq n\colon  d(S^i(z),S^i(y))<r \right\}
\end{equation*}

Then for $\mu$-almost every $y\in Y$
\begin{equation*}
\lim_{r\to 0} \liminf_{n\to\infty} -\frac{\log \mu(B_r^{(-n,n)}(y))}{2n+1}\geq h_\mu(S)
\end{equation*}
\end{thm}
\begin{proof}
We prove that for all $\delta>0$ there exists $\varepsilon>0$ such that $\mu$-almost every $y\in Y$ for all $n$ large enough and $r<\varepsilon$
\begin{equation*}
\log \mu(B_r^{(-n,n)}(y))\geq \exp\left[-(2n+1)(h_\mu(S)-\delta)\right]
\end{equation*}

\paragraph{Constructing a partition with $\mu$-continuous atoms and high entropy.}
Recall that $C\subseteq Y$ is called a $\mu$-continuity set if $\mu(\partial C)=0$.
For any point $y\in Y$ all but countable many of the balls $\left\{B_r(y) \mid r\in \mathbb{R}_{>0} \right\}$ are $\mu$-continuity sets. Hence for all $y\in Y$ there is a countable dense subset of $\mathbb{R}_{>0}$
\begin{equation*}
R_y=\left\{r_y^1,r_y^2,\ldots\right\}
\end{equation*}
such that $B_{r_y^i}(y)$ is a $\mu$-continuity set for all $i\in\mathbb{N}$.

Let $\left\{\Omega_k = B_{\rho_k}(y_0) \right\}_{k=1}^\infty$ be a sequence of $\mu$-continuous  balls centered at a common point $y_0$ with radii $\rho_k\to_{k\to\infty}\infty$. 

Let $\{y_j\}_{j=1}^\infty$ be a dense countable subset of $X$. Define the finite partition $\xi_l$ to be composed of all the mutual refinements of the sets 
\begin{equation*}
\left\{ B_{r^i_{y_j}} \cap \Omega_k \mid 1\leq i,j,k \leq l \right\} \cup \left\{ Y\setminus \Omega_l \right\}
\end{equation*}

For every open ball $B(y,r) \subset Y$ we can find a subsequence $\left\{y_{j_m}\right\}_{m=1}^\infty$ such that $d(y_{j_m}, y)<1/m$. For each $m\in\mathbb{N}$ let $i_m$ be such that $|r^{i_m}_{y_{j_m}}-r|<1/m$. Let $k\in\mathbb{N}$ such that $B(y,r+2)\subset \Omega_k.$. The ball $B_m\coloneqq B_{r^{i_m}_{y_{j_m}}}(y_{j_m})\subseteq B(y,r+2)$ belongs to $\sigma(\xi^l)$ for $l=\max\left\{i_m, j_m, k\right\}$. We have
\begin{equation*}
B(y,r)=\bigcap_{M=1}^\infty \bigcup_{m=M}^\infty B_m
\end{equation*}
Hence the $\sigma$-algebra generated by all the partitions $\left\{\xi_l\right\}_{l=1}^\infty$ is $\mathcal{Y}$.

These partitions satisfy $\sigma(\xi_l)\nearrow_{l\to\infty} \mathcal{Y}$ hence $h_\mu(S,\xi_l)
\to_{l\to_\infty} h_\mu(S)$. Moreover, all the subsets of the partitions $\left\{\xi_l \right\}_{l=1}^\infty$ are $\mu$-continuity sets.

Fix $\xi=\xi_l$ for $l$ large enough so that 
\begin{equation}\label{eq:xi-entropy}
h_\mu(\xi, S)>h_\mu(S)-\delta/4
\end{equation}

The proof continues from here using the same method as \cite{BrinKatok}; we provide the details for completeness sake.
\paragraph{Picking up good atoms}
For a set $C\subseteq Y$ and $\varepsilon>0$ define 
\begin{equation*}
C^{(\varepsilon)}\coloneqq\left\{y\in Y| d(y, C)\leq\varepsilon \right\}
\end{equation*} 
Set $E\coloneqq\bigcup_{W\in \xi} \partial W $. Because $\mu(E)=0$ we have $\mu\left(E^{(\varepsilon)}\right)\to_{\varepsilon\to 0} 0$. 

For any $y\in Y$ let 
\begin{equation*}
I_n(y)\coloneqq\left\{-n\leq i \leq n\mid S^{i}y\not\in E^{(\varepsilon)} \right\}
\end{equation*}
Define $\zeta_n(y) \coloneqq \bigcap_{i\in I_n(y)} S^{-i}\left(\xi\right)(y)$ 

\paragraph{The relation between $\zeta_n(y)$ and $B_\varepsilon^{(-n,n)}(y)$}
The set $\zeta_n(y)$ consists of all the points whose orbit in the times $i\in I_n(y)$ visits the same atoms of $\xi$ as $y$. If $y'\in B_\varepsilon^{(-n,n)}(y)$ then in all the times $-n\leq i \leq n$ we have $S^i(y')\in B_\varepsilon(S^i(y))$. For any $i\in I_n(y)$ we have $B_\varepsilon(S^i(y))\subseteq \xi(y)$, hence $S^i(y')\in \xi(a)$. This proves that $B_\varepsilon^{(-n,n)}(y)\subseteq \zeta_n(y)$.

Thus we need only prove for $\mu$-almost every $y\in Y$
\begin{equation*}
\lim_{\varepsilon\to 0} \liminf_{n\to\infty} \frac{-\log\mu\left(\zeta_n(y)\right)}{2n+1}\geq h_\mu(S)-\delta
\end{equation*}

\paragraph{Estimating the number of typical sets $\zeta_n(y)$ containing a common point} 
Define
\begin{equation*}
\mathcal{Z}_n\coloneqq \left\{\zeta_n(y) \mid y\in Y\textrm{ and } |I_n(y)|>(2n+1)(1-2\mu(E^{(\varepsilon)})) \right\}
\end{equation*}  
The family $\mathcal{Z}_n$ is a collection of subsets of $Y$. The sets $\zeta_n(y)$ in $\mathcal{Z}_n$ are those for which $y$ visits the set $E^{(\varepsilon)}$ in the typical way. Indeed, applying the ergodic theorem to the function $\chi_{E^{(\varepsilon)}}$ we see that for $\mu$-almost every $y\in Y$ one has that $\zeta_n(y)\in \mathcal{Z}_n$ for all large enough $n$.

The sets in $\mathcal{Z}_n$ are not necessarily mutually disjoint. Fix $y_0\in Y$, how many sets in $\mathcal{Z}_n$ contain it? If $y_0\in \zeta_n(y)$ then for all $i\in I_n(y)$ we have that $\xi(S^i y)=\xi(S^i y_0)$. Hence each such $\zeta_n(y)$ is the intersection of at least $(2n+1)(1-2\mu(E^{(\varepsilon)}))$ sets from the collection
\begin{equation*}
\left\{S^{-i}(\xi)(y_0)\mid -n\leq i\leq n \right\}
\end{equation*}
This implies that the number of sets in $\mathcal{Z}_n$ containing $y$ is bounded above by
\begin{align*}
F_n(\varepsilon)
&\coloneqq\sum_{m=(2n+1)(1-2\mu(E^{(\varepsilon)}))}^{2n+1} \binom{2n+1}{m}
=\sum_{m=0}^{(2n+1)2\mu(E^{(\varepsilon)})} \binom{2n+1}{m}\\
&\leq \exp\left[(2n+1)H(2\mu(E^{(\varepsilon)}))\right]
\end{align*}
Where $H(a)=-a\log a -(1-a)\log(1-a)$ is the binary entropy function in natural base.
What is important for us is that the bound is uniform in $y$ and 
\begin{equation}\label{eq:covering-eps}
H(2\mu(E^{(\varepsilon)}))\to_{\varepsilon\to0} 0
\end{equation}

\paragraph{Estimating the number of typical sets $\zeta_n(y)$  having atypically big measure} 
We want to estimate how many sets there are in the family
\begin{equation*}
\mathcal{Z}^\mathrm{big}_n\coloneqq \left\{C\in \mathcal{Z}_n \mid \mu(C)> \exp\left[-(2n+1)(h_\mu(S)-\delta)\right] \right\}
\end{equation*}
The sets in $\mathcal{Z}^\mathrm{big}_n$ are those having measure larger that what we expect asymptotically. Because $\mu$-almost every $y\in Y$ has $\zeta_n(y)\in \mathcal{Z}_n$ for all $n$ large enough,
we deduce that for $\mu$-almost every $y\in Y$ for all $n$ large enough if 
\begin{equation}\label{eq:mu-zeta-n}
\mu(\zeta_n(y))>\exp\left[-(2n+1)(h_\mu(S)-\delta)\right]
\end{equation}
then $\zeta_n(y)\in\mathcal{Z}^\mathrm{big}_n$.

We now bound the size of $\mathcal{Z}^\mathrm{big}_n$.
\begin{align*}
|\mathcal{Z}^\mathrm{big}_n|\exp\left[-(2n+1)(h_\mu(S)-\delta)\right] &\leq \sum_{C\in \mathcal{Z}^\mathrm{big}_n} \mu(C)=\int \sum_{C\in \mathcal{Z}^\mathrm{big}_n} \chi_C(y) \dif\mu (y)\\
&\leq \int F_n(\varepsilon) \dif\mu(y)=F_n(\varepsilon)
\end{align*}
Hence 
\begin{equation}\label{eq:ZnBig-bound}
|\mathcal{Z}^\mathrm{big}_n|\leq \exp\left[(2n+1)\left(h_\mu(S)-\delta+H(2\mu(E^{(\varepsilon)}))\right)\right] 
\end{equation}

\paragraph{Concluding the proof}
Define the auxiliary partition 
\begin{equation*}
\eta\coloneqq \left\{ W\cap E^{(\varepsilon)} \mid W\in \xi \right\} \cup \left\{ Y\setminus E^{(\varepsilon)} \right\}
\end{equation*}
The number of atoms in $\eta$ does not depend on $\varepsilon$ but it has one atom, $Y\setminus E^{(\varepsilon)}$, whose measure goes to $1$ when $\varepsilon\to 0$, hence 
\begin{equation}\label{eq:eta-bound}
h_\mu(S,\eta)\leq h_\mu(\eta)\to_{\varepsilon\to 0} 0
\end{equation}

For any partition $\alpha$ of $Y$ denote $\alpha^{(-n,n)}\coloneqq \bigvee_{i=-n}^n S^{-i}(\alpha)$.
By the definition of $\zeta_n$ we see for all $y\in Y$ that 
\begin{align}
\zeta_n(y) &\cap \eta^{(-n,n)}(y)\subseteq \xi^{(-n,n)}(y)
\label{eq:zeta-eta-bound}\\
&\implies \mu\left(\zeta_n(y) \cap \eta^{(-n,n)}(y)\right) \leq \mu\left(\xi^{(-n,n)}(y) \right)\nonumber
\end{align}

Define
\begin{align*}
\mathcal{X}_n&=\left\{\xi^{(-n,n)}(y) \mid y\in Y \textrm{ and } \mu (\xi^{(-n,n)}) < \exp\left[-(2n+1)(h_\mu(S,\xi)-\delta/4)\right] \right\}\\
\mathcal{H}_n&=\left\{\eta^{(-n,n)}(y) \mid y\in Y \textrm{ and }  \mu (\eta^{(-n,n)} )> \exp\left[-(2n+1)(h_\mu(S,\eta)+\delta/4)\right] \right\}
\end{align*}
These are families of typical atoms for $\xi^{(-n,n)}$ and $\eta^{(-n,n)}$ respectively. By the Shannon-MacMillan-Breiman for $\mu$-almost every $y\in Y$ we have for all $n$ large enough $\xi^{(-n,n)}(y)\in \mathcal{X}_n$ and $\eta^{(-n,n)}(y)\in \mathcal{H}_n$. All the sets in  $\mathcal{H}_n$ are mutually disjoint, thus
\begin{equation}\label{Hn-bound}
|\mathcal{H}_n| \leq \exp\left[(2n+1)(h_\mu(S,\eta)+\delta/4)\right] 
\end{equation}

Finally let
\begin{align*}
\mathcal{B}_n=\Big\{\zeta_n(y)\cap\eta^{(-n,n)}(y) \mid y\in Y &\textrm{, } \xi^{(-n,n)}(y)\in\mathcal{X}_n \\
&\textrm{, }  \eta^{(-n,n)}(y)\in\mathcal{H}_n \textrm{ and } \zeta_n(y)\in \mathcal{Z}^\mathrm{big}_n\Big\}
\end{align*} 
For any $C\in \mathcal{B}_n$ we have by \eqref{eq:zeta-eta-bound} and the definition of $\mathcal{X}_n$
\begin{equation}\label{eq:zeta-eta-bound2}
\mu(C)\leq \exp\left[-(2n+1)(h_\mu(S,\xi)-\delta/4)\right] 
\end{equation}

We can now give the main bound on the measure of the union of all the sets in $\mathcal{B}_n$ multiplying \eqref{eq:ZnBig-bound}, \eqref{Hn-bound} and \eqref{eq:zeta-eta-bound2}
\begin{align*}
&\mu\left({\bigcup_{C\in\mathcal{B}_n}}C \right) 
\leq |\mathcal{Z}^\mathrm{big}_n| |\mathcal{H}_n| \exp\left[-(2n+1)(h_\mu(S,\xi)-\delta/4)\right]\\ 
&\leq \exp\left[-(2n+1)\left(h_\mu(S,\xi)-h_\mu(S)-h_\mu(S,\eta)+\delta/2+H(2\mu(E^{(\varepsilon)}))\right)\right]\\
&\leq \exp\left[-(2n+1)\left(\delta/4-h_\mu(S,\eta)+H(2\mu(E^{(\varepsilon)}))\right)\right]
\end{align*}
By \eqref{eq:covering-eps} and \eqref{eq:eta-bound} for any $\varepsilon>0$ small enough we have 
\begin{equation*}
|h_\mu(S,\eta)-H(2\mu(E^{(\varepsilon)}))|<\delta/8
\end{equation*}
and
\begin{equation*}
\mu\left({\bigcup_{C\in\mathcal{B}_n}}C \right) <  \exp\left[-(2n+1)\delta/8\right]
\end{equation*}
\end{proof}
As this bound is summable, by Borel-Cantelli $\mu$-almost every $y\in Y$ belongs to no more then finitely many of these unions. Hence for $\mu$-almost every $y\in Y$ one of the three conditions appearing in the definition of $\mathcal{B}_n$ occurs only finitely many times, but that first two occur $\mu$-almost every $y$ for all $n$ large enough by the Shannon-McMillan-Breiman theorem. Hence $\mu$-almost every $y\in Y$ has $\zeta_n(y)\not\in \mathcal{Z}^\mathrm{big}_n$ for all $n$ large enough, but as we have discussed before this implies by the ergodic theorem that for $\mu$-almost every $y\in Y$ inequality \eqref{eq:mu-zeta-n} holds for all $n$ large enough.

\section{Homogeneity of the Haar Measure}
The following is standard, at least in the split case. We provide a proof for completeness sake.
\begin{prop}\label{prop:homogeneous}
Let $G=G_0\times G_p$ be a product of locally compact groups and $\Gamma<G$ a lattice. Suppose that $G_p=\mathbf{G}(\mathbb{Q}_p)$ for 
$\mathbf{G}$ a linear reductive group defined over $\mathbb{Q}$ and $p$ a finite rational prime. 

Let $a\in G_p$ be semisimple. By abuse of notation denote by $m$ the Haar measure on $G$ and $\lfaktor{\Gamma}{G}$ and normalize $m$ on $G$ so it induces a probability measure on $\lfaktor{\Gamma}{G}$.
Then there is a base for the topology of $G$ around the identity consisting of $(a,h_m(a))$-homogeneous neighborhoods.
\end{prop}
\begin{proof}
It suffices to produce an $(a,h_m(a))$-homogeneous base for the topology of $G_p$. In order to conjugate by $a$ instead of $a^{-1}$ we prove the theorem for $a^{-1}$-Bowen balls. Evidently, $B_{a^{-1}}^{(s,t)}=B_{a}^{(-t,-s)}$ and $h_m(a^{-1})=h_m(a)$.

We consider $\mathbf{G}$ as defined over $\mathbb{Q}_p$, i.e.\  replace $\mathbf{G}$ with its extension of scalars to $\mathbb{Q}_p$.
Let $\mathbf{A}<\mathbf{G}$ be a maximal torus whose $\mathbb{Q}_p$ points include $a$ and denote by $\mathbb{L}$ the splitting field of $\mathbf{A}$, it is a finite Galois field extension of $\mathbb{Q}_p$. Let $\mathfrak{G}\coloneqq \Gal(\mathbb{L}/\mathbb{Q}_p)$ and denote by $e$ the ramification index of $\mathbb{L}/\mathbb{Q}_p$. Let $w$ be the single place of $\mathbb{L}$ extending $p$. 

Let $\mathfrak{g}$ be the Lie algebra $\mathbf{G}$ and denote $\mathfrak{g}_p\coloneqq \mathfrak{g}(\mathbb{Q}_p)$ and $\mathfrak{g}_w\coloneqq \mathfrak{g}(\mathbb{L})$. The $\mathbb{L}$-vector space $\mathfrak{g}_w$ is equipped with a Galois structure for $\mathfrak{G}$ such that $\mathbb{Q}_p$-vector space $\mathfrak{g}_p$ is the set of fixed points for the action of $\mathfrak{G}$.

Consider the decomposition of $\mathfrak{g}_w$ into $\Ad_a$-eigenspaces with eigenvalue $\lambda$
\begin{equation*}
\mathfrak{g}_w=\bigoplus_{\lambda} \mathfrak{g}_{w,\lambda}
\end{equation*}
Let $[\lambda]$ be the Galois orbit of $\lambda$. Denote the sum of all eigenspaces with eigenvalue Galois conjugate to $\lambda$ by
\begin{equation*}
\mathfrak{g}_{w,[\lambda]}=\bigoplus_{\sigma\in\mathfrak{G}} \mathfrak{g}_{w,\sigma(\lambda)}
\end{equation*}
The space $\mathfrak{g}_{w,[\lambda]}<\mathfrak{g}_w$ is $\mathfrak{G}$-invariant hence it defined over $\mathbb{Q}_p$, let $\mathfrak{g}_{[\lambda]}<\mathfrak{g}_p$ be the subspace of $\mathfrak{G}$ fixed vectors, then $\mathfrak{g}_{[\lambda]}\otimes_{\mathbb{Q}_p} \mathbb{L}=\mathfrak{g}_{w,[\lambda]}$.

Let $X_{1,[\lambda]},\ldots,X_{k,[\lambda]}$ be a $\mathfrak{G}$-stable eigenbase for $\mathfrak{g}_{w,[\lambda]}$. We construct an $\mathcal{O}_\mathbb{L}$ lattice in $\mathfrak{g}_{w,[\lambda]}$
\begin{equation*}
L_{w,[\lambda]}=\oplus_{i=1}^k X_{i,[\lambda]} \mathcal{O}_\mathbb{L}
\end{equation*}
If $e$ is the ramification index then there exists $F_{[\lambda]}\in\mathbb{Q}_p$ such that $\lambda^e=F_{[\lambda]} \omega$ for some $\omega\in\mathcal{O}_\mathbb{L}^\times$. Any $\sigma\in\mathfrak{G}$ acts by $\sigma(\lambda^e)=F_{[\lambda]} \sigma(\omega)$ with $\sigma(\omega)\in\mathcal{O}_\mathbb{L}^\times$. Thus for all $n\in\mathbb{Z}$
\begin{equation*}
{\Ad}_{a^{en}} (L_{w,[\lambda]})= F_{[\lambda]}^n L_{w,[\lambda]}
\end{equation*}

The lattice $L_{w,[\lambda]}$ is invariant under $\mathfrak{G}$, let $L_{[\lambda]}$ be the set of $\mathfrak{G}$ fixed points in $L_{w,[\lambda]}$. By definition $L_{[\lambda]}=\mathfrak{g}_{[\lambda]}\cap L_{w,[\lambda]}$. This is an intersection of a compact open subset of $\mathfrak{g}_{w,[\lambda]}$ with a subspace, hence $L_{[\lambda]}$
is a compact open subset of $\mathfrak{g}_{[\lambda]}$. Moreover, $L_{[\lambda]}$ is clearly a $\mathbb{Z}_p$ submodule of the finitely generated $\mathbb{Z}_p$ module $L_{w,[\lambda]}$. Because $\mathbb{Z}_p$ is Noetherian we conclude that $L_{[\lambda]}$ is a lattice of full rank in $\mathfrak{g}_{[\lambda]}$.

We now show for all $n\in\mathbb{Z}$ that ${\Ad}_{a^{en}} (L_{[\lambda]})= F_{[\lambda]}^n L_{[\lambda]}$. Fix $n\in\mathbb{Z}$. The two $\mathbb{L}$-linear automorphisms ${\Ad}_{a^{en}}$ and $F_{[\lambda]}^n \Id$ of $\mathfrak{g}_{w,[\lambda]}$ are defined over $\mathbb{Q}_p$, hence ${\Ad}_{a^{en}} (L_{[\lambda]})$ and $F_{[\lambda]}^n L_{[\lambda]}$ are the sets of Galois fixed points in ${\Ad}_{a^{en}} (L_{w,[\lambda]})= F_{[\lambda]}^n L_{w,[\lambda]}$. This proves ${\Ad}_{a^{en}} (L_{[\lambda]})= F_{[\lambda]}^n L_{[\lambda]}$.

Let $|\cdot|_w$ be an absolute value corresponding to $w$ on $\mathbb{L}$. Galois conjugate eigenvalues have the same absolute value.
If $|\lambda|_w\geq 1$ then for any $t_2\geq t_1$ we have $\Ad_{a^{t_1}} L_{w,[\lambda]}\subseteq \Ad_{a^{t_2}} L_{w,[\lambda]}$, by considering 
$\mathfrak{G}$ fixed points this implies also 
$\Ad_{a^{t_1}} L_{[\lambda]}\subseteq \Ad_{a^{t_2}} L_{[\lambda]}$. Similarly, if $|\lambda|_w\leq 1$ then for all $s_2\leq s_1$ we have $\Ad_{a^{s_1}}L_{[\lambda]}\subseteq \Ad_{a^{s_2}} L_{[\lambda]}$. If $|\lambda|_w=1$ then $L_{[\lambda]}$ is invariant under $\Ad_a$.

Split the eigenvalue classes $[\lambda]$ into three groups
\begin{align*}
&\Lambda_+\coloneqq \left\{[\lambda] \mid |\lambda|_w>1 \right\}\\
&\Lambda_0\coloneqq \left\{[\lambda] \mid |\lambda|_w=1 \right\}\\
&\Lambda_-\coloneqq \left\{[\lambda] \mid |\lambda|_w<1 \right\}
\end{align*}
Define accordingly
\begin{align*}
&\mathfrak{L}_+ \coloneqq \bigoplus_{[\lambda]\in \Lambda_+} L_{[\lambda]}\\
&\mathfrak{L}_0 \coloneqq \bigoplus_{[\lambda]\in \Lambda_0} L_{[\lambda]}\\
&\mathfrak{L}_- \coloneqq \bigoplus_{[\lambda]\in \Lambda_-} L_{[\lambda]}
\end{align*}
Set also $\mathfrak{g}_+$, $\mathfrak{g}_0$ and $\mathfrak{g}_-$  to be the corresponding subspaces of $\mathfrak{g}_p$.

Fix $\mathfrak{U}\subset\mathfrak{g}_p$ a neighborhood of $0$ on which $\exp$ is an isometry.
Let $l\in\mathbb{N}$ be large enough so that $p^l L_{[\lambda]}\subseteq\mathfrak{U}$ for all $[\lambda]$. Define the zero neighborhood $\mathfrak{B}_l\coloneqq p^l\mathfrak{L}_+\oplus p^l\mathfrak{L}_0\oplus p^l\mathfrak{L}_- \subseteq\mathfrak{U}$ and set $B_l\coloneqq\exp\mathfrak{B}_l$. The sets $B_l$ for all $l$ large enough form a base for the topology around the identity and we now prove that they possess the required homogeneity property.

Using the inclusion relation between $a$-conjugates of lattices $L_{[\lambda]}$ and the action of $\Ad_{a^{en}}$ on these lattices we deduce
for $ s\leq es'\leq 0\leq et' \leq t$
\begin{align*}
\mathfrak{B}_l^{(s,t)}\coloneqq \bigcap_{s\leq n\leq t} {\Ad}_{a^{n}}\left(\mathfrak{B}_l\right)
\supseteq p^l\mathfrak{L}_0 \oplus p^l\bigoplus_{[\lambda]\in\Lambda^+} F_{[\lambda]}^{t'} L_{[\lambda]}  
\oplus p^l\bigoplus_{[\lambda]\in\Lambda^-} F_{[\lambda]}^{s'} L_{[\lambda]}
\end{align*}
hence 
\begin{align*}
\vol(\mathfrak{B}_l^{(s,t)})
&\geq \vol(\mathfrak{B}_{l}) \left(\prod_{[\lambda]\in\Lambda^+} |F_{[\lambda]}|_p\right)^{s'}
\left(\prod_{[\lambda]\in\Lambda^-} |F_{[\lambda]}|_p\right)^{t'}\\
&= \vol(\mathfrak{B}_{l})|\det{\Ad}_{a^e}\restriction_{\mathfrak{g}_+}|_p^{s'} |\det{\Ad}_{a^e}\restriction_{\mathfrak{g}_-}|_p^{t'}
\end{align*}

The determinants of the restriction of $\Ad_a$ to $\mathfrak{g}_+$ and $\mathfrak{g}_-$ are reciprocal because each character of $\mathbf{A}$ that is a root has its negative as a root as well. Fixing $s\leq0\leq t$ and taking $s',t'$ extremal we have $\vol(\mathfrak{B}_l^{(s,t)})\geq \vol(\mathfrak{B}_{l}) \exp\left[(t-s+2e)\log\det|\Ad_a\restriction_{\mathfrak{g}_-}|_p\right]$. Because $\exp$ is an isometry on $\mathfrak{U}$ this implies a similar volume estimate for $B_l^{(s,t)}$.

By \cite[\S 7.9]{Pisa} one has $h_m(a)=-\log\det|\Ad_a\restriction_{\mathfrak{g}_-}|_p$.
\end{proof}
\begin{remark}
It is easy to deduce from the proof a lower bound on $B_l^{(s,t)}$ with the same exponent $-\log\det|\Ad_a\restriction_{\mathfrak{g}_-}|_p$. Using Theorem \ref{thm:bk-second-countable} and \cite[Proposition 3.2]{ELMVPeriodic} that would imply $h_m(a)=-\log\det|\Ad_a\restriction_{\mathfrak{g}_-}|_p$ without using leafwise measure as in  \cite{Pisa}.
\end{remark}

\bibliographystyle{alpha}
\bibliography{large_deviations}

\end{document}